\documentclass[11pt,final]{./Siam/siamltex}

\usepackage{amsmath,mathtools}
\usepackage{alg}
\usepackage{xcolor}

\usepackage{amsmath,amssymb,amscd,subfigure}
\usepackage{amsfonts,graphicx,psfrag,esint}

\renewcommand{\Re}{\mathbb{R}}

\newcommand{\half}{\frac{1}{2}}
\newcommand{\weak}{\rightharpoonup}

\newcommand{\embed}{\hookrightarrow}

\newcommand{\vph}{\vphantom{A^{A}_{A}}}
\newcommand{\sst}{\,\mid\,}

\newcommand{\eqnref}[1]{(\ref{eqn#1})}

\newcommand{\uhat}{\hat{u}}
\newcommand{\vhat}{\hat{v}}

\newcommand{\Omegabar}{\bar{\Omega}}

\newcommand{\ut}{\tilde{u}}

\newcommand{\bfg}{{\bf g}}

\newcommand{\bfn}{{\bf n}}

\newcommand{\bfp}{{\bf p}}
\newcommand{\bfq}{{\bf q}}

\newcommand{\bfG}{{\bf G}}

\newcommand{\calL}{{\cal L}}

\newcommand{\calP}{{\cal P}}

\newcommand{\calT}{{\cal T}}

\newcommand{\norm}[1]{\| {#1} \|}
\newcommand{\Hone}{{H^1(\Omega)}}
\newcommand{\hone}[1]{\norm{#1}_\Hone}

\newcommand{\Hdiv}{{H(\Omega;div)}}

\newcommand{\Lone}{{L^1(\Omega)}}
\newcommand{\lone}[1]{\norm{#1}_\Lone}
\newcommand{\Ltwo}{{L^2(\Omega)}}
\newcommand{\ltwo}[1]{\norm{#1}_\Ltwo}

\newcommand{\Linf}{{L^\infty(\Omega)}}

\newcommand{\Lp}{{L^p(\Omega)}}
\newcommand{\lp}[1]{\norm{#1}_\Lp}

\newcommand{\Lq}{{L^q(\Omega)}}

\newcommand{\rmdiv}{\,\mathrm{div}}

\newcommand{\delTh}{\partial \calT_h}
\newcommand{\Honeh}{{H^1(\calT_h)}}
\newcommand{\honeh}[1]{\norm{#1}_\Honeh}
\newcommand{\Ltwoh}{{L^2(\calT_h)}}
\newcommand{\ltwoh}[1]{\norm{#1}_\Ltwoh}

\newcommand{\mbfG}{{\mathbf G}}
\newcommand{\Lr}{{L^r(\Omega)}}

\newcommand{\calPell}{\calP_\ell}
\newcommand{\lranglenorm}[1]{\langle{#1}, {#1}\rangle}
\newcommand{\lranglenormElement}[1]{\|{#1}\|_{L^2(\partial K)}}

\begin{document}

\bibliographystyle{siam}

\title{Stability and Convergence of HDG Schemes Under Minimal Regularity}

\author{Jiannan Jiang\thanks{Department of Mathematical Sciences, Carnegie
    Mellon University, Pittsburgh, PA 15213.} \and
  Noel J. Walkington\thanks{Department of Mathematical Sciences,
    Carnegie Mellon University, Pittsburgh, PA 15213. Supported in
    part by National Science Foundation Grant DMS--2012259.  This work was
    also supported by the NSF through the Center for Nonlinear
    Analysis.} \and
    Yukun Yue\thanks{Department of Mathematical Sciences, University 
    of Wisconsin-Madison, Madison, WI, 53706.}
}

\date{\today}
\maketitle

\begin{abstract}
  Convergence and compactness properties of approximate solutions to
  elliptic partial differential equations computed with the hybridized
  discontinuous Galerkin (HDG) scheme of Cockburn, Gopalakrishnan, and
  Sayas (Math. Comp., 79 (2010), pp.~1351--1367) are
  established. While it is known that solutions computed using this
  scheme converge at optimal rates to smooth solutions, this does not
  establish the stability of the method or convergence to solutions
  with minimal regularity. The compactness and convergence results
  show that the HDG scheme can be utilized for the solution of
  nonlinear problems and linear problems with non--smooth coefficients
  on domains with reentrant corners.
\end{abstract}

\begin{keywords}
HDG scheme, stability, convergence, finite element.
\end{keywords}

\begin{AMS} 65N12, 65N30 \end{AMS}

\section{Introduction}
Among the vast number of discontinuous Galerkin (DG) schemes for
elliptic equations \cite{ArBrCoMa02,castillo2000priori,CoBeKa05,Co00,CoSh98,HeWa08,Ri08,WaYe13,WaYe14},
the hybridized discontinuous Galerkin (HDG) scheme of Cockburn
et al. \cite{cockburn2004characterization,CoGoSa10,cockburn2012conditions,CoSh98,efendiev2015spectral,FuQi15,HeWa08,Hu20,li2016analysis,sevilla2016tutorial}
is one of the few schemes satisfying the following desirable
properties:
\begin{itemize}
\item The variational structure is preserved in the sense that the
  approximation of a self adjoint operator is self adjoint; this is
  essential for approximating wave equations or gradient flows.

\item The stability (penalty) parameters for the HDG scheme are
  explicit. For many DG schemes, parameters depend upon obscure
  quantities such as Poincar\'e  or trace constants of the domain.

\item The scheme exhibits optimal approximation of smooth solutions,
  and improved rates for lower order norms can be established
  using duality arguments upon post--processing.
  
\item Hybridizing guarantees that HDG schemes have complexity
  comparable to continuous Galerkin and other DG schemes.

\item Piecewise constant approximations of both the function and its
  gradient are admissible. This is very useful for constructing
  conforming approximations of nonlinear problems of the form
  $u: \Omega \subset \Re^d \to \Re$,
  $$
  e - \Delta u = f,
  \qquad e \in \partial \psi(u),
  $$
  where $\psi:\Re \rightarrow \Re$ is convex, and minima of the
  corresponding variational problem;
  $$
  I(u) = \int_\Omega \Psi(u,\nabla u),
  \quad \text{ with } \quad
  \Psi(u,\bfp) = \psi(u) + |\bfp|^2 - f u.
  $$

  Euler approximations of the classical Stefan problem are a
  prototypical example of such equations \cite{BeMaHeRo12,RuWa96}.
\end{itemize}
In \cite{CoGoSa10,cockburn2012conditions} 
Cockburn et al. proved that if the solution of the underlying 
elliptic equation is smooth, then their HDG approximations converge 
at optimal rates. It is well known that convergence to
smooth solutions does not guarantee stability \cite{BaNa97}, and
prototypically solutions of nonlinear problems are not smooth.
Cockburn et al. did not consider stability, and stability for their
HDG scheme is not addressed in texts such as \cite{Gu21}.

Below trace, compactness, and convergence properties of the HDG
spaces are developed and used to establish stability and convergence
of HDG approximations of the canonical linear elliptic problem;
\begin{equation} \label{eqn:pde}
-\rmdiv(A \nabla u) = f, \quad \text{ in  $\Omega \subset \Re^d$},
\qquad u|_{\Gamma_0} = u_0,
\qquad A \nabla u.\bfn|_{\Gamma_1} = g,
\end{equation}
where $\partial \Omega = \Gamma_0 \cup \Gamma_1$ is a partition of the
boundary of $\Omega$ and $A:\Omega \rightarrow \Re^{d \times d}$ is
elliptic. Specifically, if $\calT_h$ is a triangulation of $\Omega$,
and $\calPell(\calT_h)$ and $\calPell(\delTh)$ denote the spaces of
(discontinuous) piecewise polynomials of degree $\ell \geq 0$ over
$\calT_h$ and the $d-1$ faces of $\calT_h$ respectively, the HDG
approximations take the (mixed) form,
$(u_h, \uhat_h, \bfp_h) \in \calPell(\calT_h) \times \calPell(\delTh)
\times \calPell(\calT_h)^d$,
\begin{subequations} \label{eqn:wp}
\begin{gather}
  \langle u_h-\uhat_h, v_h-\vhat_h \rangle
  - \left(\rmdiv(\bfp_h), v_h \vph\right)
  + \langle \bfp_h.\bfn, \vhat_h \rangle
  = (f,v_h) + \langle g, \vhat \rangle_{\Gamma_1} \label{eqn:wpU} \\
  \left(A^{-1} \bfp_h, \bfq_h \vph\right)
  + \left(\rmdiv(\bfq_h), u_h \vph\right)
  - \langle \bfq_h.\bfn, \uhat_h \rangle
  = 0, \label{eqn:wpP} 
\end{gather}
\end{subequations}
for all
$(v_h, \vhat_h, \bfq_h) \in \left\{ \calPell(\calT_h) \times
  \calPell(\delTh) \times \calPell(\calT_h)^d \sst \vhat_h|_{\Gamma_0} = 0
  \right\}$, and
\begin{equation}
  \langle \uhat_h, \vhat_h \rangle_{\Gamma_0}
  = \langle u_0, \vhat_h \rangle_{\Gamma_0}, \label{eqn:wpUo}
\end{equation}
for all $\vhat_h \in \calPell(\delTh \cap \Gamma_0)$. In the above
\begin{equation} \label{eqn:pairings}
(u,v) \equiv \sum_{K \in \calT_h} \int_K u \, v, \qquad
\langle u,v \rangle \equiv \sum_{K \in \calT_h}
\int_{\partial K} u \, v, \qquad
\langle u,v \rangle_{\Gamma_0} \equiv \sum_{k \in \delTh \cap \Gamma_0}
\int_k u \, v,
\end{equation}
are broken $L^2$ pairings. As is common for mixed formulations, we assume
that $f \in \Ltwo$ and $g \in L^2(\Gamma_1)$ which are slightly more
regular than the usual dual spaces used for the primal problem.  We
subsequently show that, when $A$ is symmetric, the HDG scheme for
both linear and nonlinear elliptic equations can be realized
as the Euler Lagrange equation of a variational problem.

As outlined above, two of the distinguishing features of the particular
HDG scheme \eqnref{:wp} are:
\begin{itemize}
\item Piecewise constant functions, $\ell=0$, are admissible.
\item The stabilization constant is explicit; specifically, the first
  term on the left of \eqnref{:wpU} has unit weight.
\end{itemize}
Frequently the stronger weight $1/h$ is utilized for the stabilization
term
\cite{Gu21,buffa2009compact,chen2023superconvergence,kirk2019analysis},
and this enters the analysis in an essential fashion. However, when
all variables are approximated with polynomials of the same degree,
only the constant weight yields an optimal rate of convergence. The
following example illustrates this for the piecewise constant HDG
scheme.

\begin{figure}
  \hspace*{-0.375in}
  \begin{tabular}{|l|cc|cc|cc|} \hline
    & \multicolumn{2}{c|}{\text{weight $\tau = 1$}}
    & \multicolumn{2}{c|}{\text{weight $\tau = h^{-1}$}}
    & \multicolumn{2}{c|}{\text{weight $\tau = h$}}
    \\ \hline
    \multicolumn{1}{|c|}{$h$}
    & $\ltwoh{u-u_h}$  & $\ltwoh{\bfp-\bfp_h}$
    & $\ltwoh{u-u_h}$  & $\ltwoh{\bfp-\bfp_h}$ 
    & $\ltwoh{u-u_h}$  & $\ltwoh{\bfp-\bfp_h}$
    \\ \hline
    1/4 & 1.458986 & 2.703363 & 0.377034 & 2.678867 & 5.835746 & 2.745128 \\
    1/8 & 0.743051 & 1.293497 & 0.129315 & 1.483995 & 5.961673 & 1.298312 \\
    1/16 & 0.354562 & 0.611884 & 0.093955 & 1.005111 & 5.695665 & 0.610560 \\
    1/32 & 0.178083 & 0.311185 & 0.093149 & 0.874985 & 5.721947 & 0.310110 \\
    1/64 & 0.088250 & 0.155298 & 0.093290 & 0.838297 & 5.672201 & 0.154572 \\
    1/128 & 0.044529 & 0.077636 & 0.093307 & 0.828456 & 5.725109 & 0.077232 \\
    \hline
  \end{tabular}
  \caption{Piecewise constant HDG approximation of
  solution to Laplace's equation with different weights.}
\label{fig:hdg0}
\end{figure}

\begin{example} \label{eg:smooth}
  To illustrate the convergence properties of the piecewise constant
  HDG approximations, we consider a manufactured solution
  $u(x,y) = \cos(2\pi x) \cos(2\pi y)$ on $\Omega = [0,1]^2$ of
  Laplace's equation with Dirichlet data on the top and bottom
  boundaries and Neumann data on the sides. Solutions were computed
  with various weightings of the stabalization term in \eqnref{:wpU},
  $$
  \tau \langle u_h-\uhat_h, v_h-\vhat_h \rangle
  - \left(\rmdiv(\bfp_h), v_h \vph\right)
  + \langle \bfp_h.\bfn, \vhat_h \rangle
  = (f,v_h) + \langle g, \vhat \rangle_{\Gamma_1}.
  $$
  Figure \ref{fig:hdg0} tabulates the errors for solutions computed on
  uniform meshes with weights $\tau = 1$, $1/h$, and $h$. It is clear that
  both $u_h$ and $\bfp_h$ converge only when the unit weight is
  utilized.
\end{example}

\subsection{Related Results}
Discontinuous Galerkin (DG) methods for elliptic problems have been
intensively studied for the past half--century, giving rise to a vast variety of schemes 
and an extensive body of literature. For this
reason, only a brief summary and closely related results are reviewed
here. Originally DG schemes were constructed using a primal formulation
of the elliptic equation \cite{Ar82,Gu21,OdBuBa98,Ri08,RiWh99,WaYe13}, and,
motivated by DG schemes for hyperbolic problems, many subsequent
developments utilized a mixed formulation involving the interelement
fluxes \cite{BrFo91, CoSh98,CoBeKa05,Co00,WaYe14}. The
function spaces for these schemes are piecewise polynomial subspaces
of $\Ltwo$ over a triangulation $\calT_h$ of the domain, and a unified
analysis of these schemes was developed in \cite{ArBrCoMa02}. 
DG schemes have been utilized and analysed for a vast number of problems
including, for example, 
Maxwell's equations 
\cite{cockburn2004locally,du2020unified}, 
Boltzmann equations 
\cite{cheng2011brief,chung2020generalized,jaiswal2019discontinuous},
multiscale problems 
\cite{efendiev2009multiscale, hou2015sparse, efendiev2022efficient},
and phase-field problems \cite{chen2023superconvergence, kay2009discontinuous,xia2007local}, among others.

The HDG scheme introduced by Cockburn et
al. \cite{nguyen2010hybridizable,cockburn2016static,cockburn2009unified,cockburn2004characterization}
is a mixed formulation of the equation with an additional scalar
valued variable defined on the $d-1$ skeleton of the triangulation
which plays the role of a trace of the solution.  In contrast with
other schemes that express the interelement flux as functions (jumps,
averages, etc.) of the variables in the adjacent elements, the HDG
fluxes on an element boundary depend only upon the variables
associated with that element and the traces on the skeleton.  For
linear problems, static condensation at the element level results in a
linear system of equations for the traces.

Solutions of elliptic equations are typically in $\Hone$, and it is
necessary to introduce penalty terms in DG formulations in order to
recover essential properties of $\Hone$ required for continuity and
coercivity of the bilinear functions and convergence under mesh
refinement.  In a series of papers, Brenner systematically developed
Poincar\'e, trace, and Korn inequalities for discontinuous finite
element spaces \cite{Br03,Br04}. A key observation was
that in many instances these inequalities are held under a bound on the
jumps on inter--element faces of the form
\begin{equation} \label{eqn:edgeWeights}
\sum_{k \in \delTh} \frac{1}{h^d} \left(\int_k [u] \right)^2 \leq C,
\qquad \text{ as opposed to } \qquad
\sum_{k \in \delTh} \frac{1}{h} \left(\int_k [u]^2 \right) \leq C,
\end{equation}
which is commonly assumed \cite{Gu21,buffa2009compact}. This distinction is
essential for the analysis below of the HDG scheme \eqnref{:wp}
for which the penalty term is unweighted.

Embedding and compactness properties of Sobolev spaces are needed for
the existence theory for nonlinear elliptic equations and variational
problems. Buffa and Ortner \cite{buffa2009compact} develop general
embedding and compactness results for spaces of discontinuous
functions. However, these are all developed under the assumption that
the jump term on the right of equation \eqnref{:edgeWeights} is
bounded, and this is essential for many of their estimates. Below we
establish embedding and compactness properties of discontinuous spaces
with the penalty term having unit weight as in the HDG scheme. When the
polynomials have degree $\ell \geq 1$, our estimates are not as sharp as
those in the continuous setting.  Specifically, while
$\Hone \embed \Lp$ with $p \leq 2 d / (d-2)$ ($p < \infty$ if $d=2$)
and compact if the inequality is strict, we show solutions of the HDG
scheme are bounded in $\Lp$ with $p \leq 2d/(d-1)$ and compact if the
inequality is strict. In addition, the trace operator
$\gamma:\Hone \rightarrow L^p(\partial\Omega)$ is continuous with
$p \leq 2(d-1)/(d-2)$ and compact if the inequality is strict;
however, in the HDG context, we only show continuity into
$L^2(\partial\Omega)$.

The distributional gradient of a discontinuous function consists of
the Lebesgue part (a.k.a. the ``broken'' gradient) and a term computed
from the jumps. When approximating nonlinear and variational problems
with DG schemes it is necessary to construct a piecewise polynomial
approximation of distributional gradients in $\Ltwo$. This was done
in \cite{buffa2009compact,FeLeNe16} for the classical DG spaces, and
for other DG schemes in \cite{Gu21}. The approximation of the
distributional gradient we introduce for the analysis of the HDG
scheme ($\bfG_h(u_h,\uhat_h) \in \calPell(\calT_h)^d$ below) is similar
to the approximations appearing in \cite{WaYe13}, but differs from the
one considered by \cite{Gu21} in that we do not require it to be the
broken gradient of a function in $\calP_{\ell+1}(\calT_h)$.

\subsection{Overview}
Section \ref{sec:pre} begins with an outline of the notations required
for the analysis of DG schemes. To establish global properties,
such as the Poincar\'e inequality, we embed the DG spaces into the
Crouzeix Raviart (CR) space. Section \ref{sec:pre} recalls the
essential properties CR spaces required below. The analysis for the
higher order HDG schemes may be viewed as a technical generalization
of the piecewise constant HDG scheme, so this simpler case is
considered separately in Section \ref{sec:hdg0}, and the general case
is analyzed in Section \ref{sec:hdg}. These sections may be read
independently. Section \ref{sec:hdv} develops the convergence theory
for HDG approximations of variational statements of nonlinear elliptic
equations.

Since convergence to smooth solutions at optimal rates has been
established by Cockburn et al.
\cite{CoGoSa10,cockburn2012conditions}, we consider the
stability and convergence of solutions of the HDG scheme assuming only
that solutions to the continuous problem are in $\Hone$. This shows
that HDG methods can be utilized for nonlinear problems exhibiting
fronts and linear problems on domains with reentrant corners or
non--smooth coefficients.

\section{Preliminaries} \label{sec:pre}
\subsection{Notation}
Below $\Omega \subset \Re^d$ will denote a connected bounded domain with
a Lipschitz boundary in dimension $d=2$ or $3$ and
$\partial \Omega = \Gamma_0 \cup \Gamma_1$ denotes a partition of the
boundary. Standard notation is adopted for the Lebesgue spaces, $\Lp$,
and the Sobolev space $\Hone$. The coefficient
$A \in \Linf^{d \times d}$ will assumed to be uniformly elliptic; that
is there exists $0 < c \leq C$ such that
$$
c |\xi|^2 \leq \xi^\top A(x) \xi\leq C |\xi|^2,
\qquad \xi \in \Re^d, \,\, x \in \Omega.
$$
Triangulations of $\Omega$ will be parameterized by the diameter $h$
of their largest simplex. To avoid unnecessary technicalities, we
assume that the restriction of $\calT_h$ to $\partial \Omega$ induces
triangulations of the boundary partition
$\partial \Omega = \Gamma_0 \cup \Gamma_1$.  Constants appearing in
the estimates below may depend upon the ellipticity constants, the
domain $\Omega$, the boundary partition
$\partial\Omega = \Gamma_0 \cup \Gamma_1$, and degree $\ell$ of
polynomials utilized in the HDG scheme. The aspect ratio of a simplex
is the ratio of radius of the largest inscribing sphere to diameter.
Typically the constants will also depend upon the maximal aspect ratio
of simplices in the mesh (the mesh aspect ratio), and this will be
explicitly stated. In particular, the constants for a regular family
of meshes \cite{BrSc08,Ci78} will be bounded independently of $h$.

The broken Lebesgue and Sobolev spaces and piecewise polynomial spaces
of degree $\ell \geq 0$ over $\calT_h$ are denoted
\begin{eqnarray*}
  \Ltwoh
  &=& \{u \in \Ltwo \sst u|_K \in L^2(K), \,\, K \in \calT_h \}, \\
  \Honeh
  &=& \{u \in \Ltwo \sst u|_K \in H^1(K), \,\, K \in \calT_h \}, \\
  \calPell(\calT_h)
  &=& \{u_h \in \Ltwo \sst u_h|_K \in \calPell(K), \,\, K \in \calT_h \},
\end{eqnarray*}
and we write
$$
\norm{u}_{\Honeh}^2 = \ltwo{u}^2 + \ltwoh{\nabla u}^2
= \sum_{K \in \calT_h} \norm{u}^2_{L^2(K)} + \sum_{K \in \calT_h} \norm{\nabla u}^2_{L^2(K)}
= \sum_{K \in \calT_h} \norm{u}^2_{H^1(K)},
$$
where the $d$--simplicies in $\calT_h$ are indexed as $K \in \calT_h$. 
The $d-1$ sub--complex of $\calT_h$ (edges 2d, faces 3d) is denoted by
$\delTh$, and the broken spaces $L^2(\delTh)$ and $\calPell(\delTh)$
are the analogs of their $d$ dimensional counterparts, and the
$d-1$--simplicies in indexed as $k \in \delTh$.

The outward directed normal at the boundary of a simplex $K \in \calT_h$ is
denoted by $\bfn$, and the dot product with $\bfp \in \Re^d$ is denoted
by $\bfp.\bfn$. The notation in equation \eqnref{:pairings} is used to
denote $L^2$ pairings of both scalar and vector valued functions. If
$u_h \in \Honeh$ then $\nabla u_h$ will denote the broken (not
distributional) gradient, and similarly if $\bfp_h \in \Honeh^d$ then
$\rmdiv(\bfp_h)$ denotes the broken divergence. Functions in these
spaces are double valued on the internal faces of $\calT_h$, and
integrals over the boundaries $\partial K$ of simplices
$K \in \calT_h$ always utilizes the trace from the interior of $K$.

\subsection{Crouzeix Raviart Space}
The Crouzeix Raviart space over a triangulation $\calT_h$ of a domain
$\Omega \subset \Re^d$ is
$$
CR(\calT_h) = \left\{u_h \in \calP_1(\calT_h) \sst
\int_k [u_h] = 0, \,\, k \in \delTh \cap \Omega \right\},
$$
where $[u_h]$ denotes the jump of $u_h$ across an internal face. The
degrees of freedom are the average value of the function on the faces
which coincides with the value at the face centroids.  The natural
lifting of functions on $\delTh$ to $CR(\calT_h)$ takes the
degrees of freedom to be the averages on each face.

\begin{definition} \label{def:Lh}
  The lifting  $\calL_h: L^1(\delTh) \rightarrow CR(\calT_h)$ is the
  function for which $\uhat_h = \calL_h(\uhat)$ has degrees of freedom
  $$
  \int_k \uhat_h = \int_k \uhat, \qquad k \in \delTh.
  $$
\end{definition}

Integration by parts gives the following lemma that characterizes the
gradients of functions in $CR(\calT_h)$.

\begin{lemma} \label{lem:crGrad}
  Let $\calT_h$ be a triangulation of a domain $\Omega \subset \Re^d$,
  and $CR(\calT_h)$ denote the Crouzeix Raviart space.
  If $u_h \in CR(\calT_h)$, then $\bfp_h = \nabla u_h$ (the broken gradient)
  if and only if $\bfp_h \in \calP_0(\calT_h)^d$ and
  $$
  (\bfp_h, \bfq_h) = \langle u_h, \bfq_h.\bfn \rangle, \qquad \bfq_h \in \calP_0(\calT_h)^d.
  $$
\end{lemma}\begin{proof}
  Let $u_h \in CR(\calT_h)$.
    
  $(\Rightarrow)$ If $\bfp_h = \nabla u_h$, then
  $\bfp_h \in \calP_0(\calT_h)^d$ and if $\bfq_h \in \calP_0(\calT_h)$,
  $$
  (\bfp_h, \bfq_h)
  = \sum_{K \in \calT_h} \int_K \bfp_h.\bfq_h
  = \sum_{K \in \calT_h} \int_K \nabla u_h.\bfq_h
  = \sum_{K \in \calT_h} \int_{\partial K} u_h \, \bfq_h.\bfn
  = \langle u_h \, \bfq_h.\bfn \rangle
  $$
  
  $(\Leftarrow)$ Let $\bfp_h \in \calP_0(\calT_h)$ satisfy the weak
  statement, and set $\bfq_h = \nabla \psi$ where
  $\psi \in \calP_1(\calT_h)$. Then
  $$
  (\bfp_h, \nabla\psi)
    = \sum_{K \in \calT_h} \int_K \bfp_h.\nabla\psi
    = \sum_{K \in \calT_h} \int_{\partial K} u_h \, (\nabla\psi.\bfn)
    = \sum_{K \in \calT_h} \int_K (\nabla u_h, \nabla\psi) 
    = (\nabla u_h, \nabla\psi)
  $$
  It follows that $\bfp_h = \nabla u_h$ since $\nabla: \calP_1(\calT_h)
  \rightarrow \calP_0(\calT_h)$ is surjective.
\end{proof}

Functions in $CR(\calT_h)$ inherit many of the usual properties of
functions in $\Hone$. Proofs of the following properties are given in
the Appendix.

\begin{lemma} \label{lem:CRproperties}
  Let $\Omega \subset \Re^d$ be a connected bounded Lipschitz domain.
  \begin{itemize}
  \item (Embedding) Let $\calT_h$ be a triangulation of $\Omega$ and
    $CR(\calT_h)$ denote the Crouzeix Raviart space. Then there exists
    a constant depending only upon the maximal aspect ratio of
    $\calT_h$ such that for all $u_h \in CR(\calT_h)$
    $$
    \lp{u_h} \leq C \honeh{u_h},
    \qquad 1 \leq p \leq \frac{2d}{d-2}
    \quad \text{ ($p < \infty$ if $d=2$)},
    $$
    and
    $$
    \norm{u_h}_{L^p(\partial\Omega)} \leq C \honeh{u_h},
    \qquad 1 \leq p \leq \frac{2(d-1)}{d-2}.
    $$    
  \item (Poincar\'e) Let $\calT_h$ be a triangulation of $\Omega$, and
    $CR(\calT_h)$ denote the Crouzeix Raviart space. If
    $\partial \Omega = \Gamma_0 \cup \Gamma_1$ is a partition of the
    boundary with $|\Gamma_0| > 0$ and each component is triangulated
    by $\delTh$, then there exists a constant depending only upon the
    maximal aspect ratio of $\calT_h$ and $\Gamma_0$ such that
    $$
    \ltwo{u_h} \leq C \ltwoh{\nabla u_h},
    \qquad u_h \in CR(\calT_h).
    $$

  \item (Compactness) Let $\{\calT_h\}_{h>0}$ be a regular family of
    triangulations of $\Omega$, and $\{CR(\calT_h)\}_{h>0}$ be the
    corresponding Crouzeix Raviart spaces. If $u_h \in CR(\calT_h)$
    and $\{\honeh{u_h} \}_{h > 0}$ is bounded, then there exists a
    subsequence (not relabeled) and $u \in \Hone$ such that
    $$
    u_h \rightarrow u \,\, \text{ in } \Ltwo, \qquad
    u_h|_{\partial \Omega} \rightarrow u|_{\partial \Omega}
    \,\, \text{ in } L^2(\partial\Omega),  \qquad
    \nabla u_h \rightharpoonup \nabla u \,\, \text{ in } \Ltwo^d.
    $$
  \end{itemize}
\end{lemma}

The liftings of functions in $\Hone$ satisfy the usual approximation
properties.

\begin{lemma} \label{lem:CRdensity}
Let $\calT_h$ be a triangulation of a bounded Lipschitz domain
$\Omega \subset \Re^d$, and $CR(\calT_h)$ denote the Crouzeix Raviart
space.  Define an embedding
$\iota_h:\Hone \embed \calP_0(\calT_h) \times \calP_0(\delTh)$ by
$\iota_h(v) = (v_h, \vhat_h)$ with
$$
v_h|_K = \frac{1}{|K|} \int_{K} v,
\quad K \in \calT_h,
\qquad \text{ and } \qquad
\vhat_h|_k = \frac{1}{|k|} \int_k v, \quad k \in \delTh.
$$
Then there exists a constant $C > 0$ depending only upon the aspect
ratio of $\calT_h$ such that for all $v \in \Hone$ 
and $(v_h,\vhat_h) = \iota_h(v)$,
\begin{itemize}
  \item $\lranglenorm{v_h - \vhat_h} \leq C h \ltwo{\nabla v}^2.$

  \item $\ltwo{\calL_h(\vhat_h) - v} + \ltwo{v_h - v}
    \leq C h |v|_{H^1(\Omega)}.$
    
  \item In addition, if $v \in H^2(\Omega)$ then
  $$
  \ltwo{\calL_h(\vhat_h) - v} + h \ltwoh{\nabla (\calL_h(\vhat_h) - v)}
  \leq C h^2 |v|_{H^2(\Omega)}.
  $$

\item $\ltwoh{\nabla \calL_h(\vhat_h)} \leq C \ltwo{\nabla v}$.

  In particular, for a regular family of meshes
  $\displaystyle \lim_{h \rightarrow 0} \ltwoh{\nabla \calL_h(\vhat_h)
    - \nabla v} = 0$
\end{itemize}
\end{lemma}

These properties follow from standard parent element calculations
and scaling so are omitted. 

\section{Piecewise Constant HDG Scheme} \label{sec:hdg0}
The HDG scheme \eqnref{:wp} with piecewise constant
functions become
$(u_h, \uhat_h, \bfp_h) \in \calP_0(\calT_h) \times \calP_0(\delTh)
\times \calP_0(\calT_h)^d$
\begin{subequations}
  \begin{align}
  \langle u_h-\uhat_h, v_h-\vhat_h \rangle
  &+ \langle \bfp_h.\bfn, \vhat_h \rangle
  = (f,v_h) + \langle g, \vhat \rangle_{\Gamma_1}
  \qquad&& (v_h, \vhat_h) \in U_h, \\
  \left(A^{-1} \bfp_h, \bfq_h \vph\right)
  &- \langle \bfq_h.\bfn, \uhat_h \rangle = 0,
  \qquad&& \bfq_h \in \calP_0(\calT_h)^d, \\
  \label{eqn:piecewise HDG boundary}
  \langle \uhat_h, \vhat_h \rangle_{\Gamma_0}
  &= \langle u_0, \vhat_h \rangle_{\Gamma_0},
  \qquad&& \vhat_h \in \calP_0(\delTh \cap \Gamma_0),
\end{align}
\end{subequations}
where $U_h = \{(v_h, \vhat_h) \in \calP_0(\calT_h) \times \calP_0(\delTh)
\sst \vhat_h|_{\Gamma_0} = 0 \}$.  Lemma \ref{lem:crGrad} shows that the
first two equations may be written as
\begin{subequations} \label{eqn:hdg0}
\begin{gather}
  \langle u_h-\uhat_h, v_h-\vhat_h \rangle
  + \left(\bfp_h, \nabla \calL_h(\vhat_h) \vph\right)
  = (f,v_h) + \langle g, \vhat \rangle_{\Gamma_1} \label{eqn:hdg0u} \\
  \left(A^{-1} \bfp_h, \bfq_h \vph\right)
  - \left(\nabla \calL_h(\uhat_h), \bfq_h \vph \right) = 0,
  \label{eqn:hdg0p}
\end{gather}
\end{subequations}
where $\calL_h: L^1(\delTh) \rightarrow CR(\calT_h)$ is the lifting
defined in Definition \ref{def:Lh}. It follows that
$\bfp_h = A_h \nabla \calL_h(\uhat_h)$, where, for each
$K \in \calT_h$, the matrix $A_h|_K$ is the inverse of the average
value of $A^{-1}$ on $K$. Eliminating $\bfp_h$ from the first equation
gives
$$
\langle u_h-\uhat_h, v_h-\vhat_h \rangle
+ \left(A_h \nabla \calL_h(\uhat_h), \nabla \calL_h(\vhat_h) \vph\right)
= (f,v_h) + \langle g, \vhat \rangle_{\Gamma_1}.
$$
The classical Lax Milgram theorem can now be used to establish
stability.  Define the norm $\norm{.}_{U_h}$ on $U_h$ by
\begin{equation}  \label{eqn:UhoNorm} 
  \norm{(u_h,\uhat_h)}^2_{U_h}
  = \lranglenorm{u_h - \uhat_h} + \ltwoh{\nabla \calL_h(\uhat_h)}^2.
\end{equation}
This is a norm on $U_h$, and semi--norm on $H^1(\calT_h) \times L^2(\delTh)$.

When $A$ is elliptic, it is immediate that the left hand side of the
weak statement is bilinear, continuous, and coercive on
$(U_h, \norm{.}_{U_h})$. The following lemma will show that the right hand
side is continuous and that the Dirichlet boundary data is achieved
in the limit as $h \rightarrow 0$.

\begin{lemma} \label{lem:consistent0}
  Let $\calT_h$ be a triangulation of a bounded Lipschitz domain
  $\Omega \subset \Re^d$ and $\Gamma_0 \subset \partial \Omega$ be
  triangulated by $\calT_h$.
  \begin{itemize}
  \item If $\vhat_h \in P_0(\delTh)$ then there exists a constant
    depending only upon the aspect ratio of $\calT_h$ such that
    $$
    \norm{\vhat_h}_{L^2(\partial\Omega)}
    \leq C  \norm{\calL_h(\vhat_h)}_{H^1(\calT_h)},
    \quad \text{ and } \quad 
    \norm{\calL_h(\vhat_h) - \vhat_h}_{L^2(\partial\Omega)}
    \leq C \ltwoh{\nabla \calL_h(\vhat_h)} \, \sqrt{h}.
    $$

  \item Set
    $
    U_h = \{(v_h, \vhat_h) \in \calP_0(\calT_h) \times \calP_0(\delTh)
    \sst \vhat_h|_{\Gamma_0} = 0 \}
    $
    and assume that $|\Gamma_0| > 0$. Then there exists a constant
    $C > 0$ depending only upon the aspect ratio of $\calT_h$ such that
    $$
    \ltwo{v_h} 
    \leq C \norm{(v_h, \vhat_h)}_{U_h},
    \quad \text{ and } \quad
    \ltwo{v_h - \calL_h(\vhat_h)} 
    \leq C \norm{(v_h, \vhat_h)}_{U_h} \sqrt{h},
    $$
    for all $(v_h, \vhat_h) \in U_h$.
  \end{itemize}
\end{lemma}

\begin{proof}
  \begin{itemize}
  \item If $k = K \cap \partial \Omega$ with $K \in \calT_h$, then
    $\vhat_h|_k = \Pi_0(\calL_h(\vhat_h)|_k)$ where
    $\Pi_0:L^2(k) \rightarrow \calP_0(k)$ is the projection onto the
    constant functions, so
    $$
    \norm{\vhat_h}_{L^2(\partial\Omega)}
    \leq \norm{\calL_h(\vhat_h)}_{L^2(\partial\Omega)}
    \leq \norm{\calL_h(\vhat_h)}_{H^1(\calT_h)}.
    $$
    In addition, the mapping from $u \in H^1(K)$ to
    $u|_k - \Pi_0(u|_k) \in L^2(k)$ vanishes when $u$ is constant, so a
    parent element calculation shows
    $$
    \norm{\calL(\vhat_h) - \vhat_h}_{L^2(k)}^2
    \leq C \norm{\nabla \calL(\vhat_h)}_{L^2(K)}^2 \, h_K,
    \, \text{ whence } \,
    \norm{\calL_h(\vhat_h) - \vhat_h}_{L^2(\partial\Omega)} \leq C
    \ltwoh{\nabla \calL_h(\vhat_h)} \, \sqrt{h}.
    $$

  \item Let $(v_h, \vhat_h) \in U_h$ and assume that $|\Gamma_0| > 0$. 

    For $u_h \in \Honeh$ let $u_\partial \in \calP_0(\calT_h)$ be the
    function taking the average over the boundary of each element,
    $$
    u_\partial|_K = \frac{1}{|\partial K|} \int_{\partial K} u_h,
    \qquad \text{ in particular, } \qquad
    \calL_h(\vhat_h)_\partial
    = \frac{1}{|\partial K|} \int_{\partial K} \vhat_h.    
    $$
    Since $u_h \mapsto u_h - u_\partial$ vanishes when $u_h \in
    \calP_0(\calT_h)$, a parent element calculation shows
    $$
    \norm{u_h - u_\partial}^2_{L^2(K)}
    \leq C \norm{\nabla u_h}^2_{L^2(K)} h^2_K,
    \qquad \text{ so } \qquad
    \ltwo{u_h - u_\partial} \leq C \ltwoh{\nabla u_h} h.
    $$
    Next,
    \begin{eqnarray*}
      \ltwo{v_h - \calL_h(\vhat_h)_\partial}^2
      &=& \sum_{K \in \calT_h} |K| (v_h - \calL_h(\vhat_h)_\partial)^2 \\
      &=& \sum_{K \in \calT_h} |K| (v_h - \frac{1}{|\partial K|}
          \int_{\partial K} \vhat_h)^2 \\
      &=& \sum_{K \in \calT_h} |K| \left(\frac{1}{|\partial K|}
          \int_{\partial K} (v_h - \vhat_h) \right)^2 \\
      &\leq& \sum_{K \in \calT_h} 
             \frac{|K|}{|\partial K|} \int_{\partial K} (v_h - \vhat_h)^2 
             \leq C h \lranglenorm{v_h - \vhat_h}.
    \end{eqnarray*}
    The bound on $v_h - \calL_h(\vhat_h)$ now follows from the
    triangle inequality,
    $$
    \ltwo{v_h - \calL_h(\vhat_h)}
    \leq \ltwo{v_h - \calL_h(\vhat_h)_\partial}
    + \ltwo{\calL_h(\vhat_h)_\partial - \calL_h(\vhat_h)}
    \leq C \norm{(v_h, \vhat_h)}_{U_h} \, \sqrt{h}.
    $$
    The Poincar\'e inequality of Lemma \ref{lem:CRproperties}
    can now be used to bound $\ltwo{v_h}$,
    $$
    \ltwo{v_h}
    \leq \ltwo{v_h - \calL_h(\vhat_h)}
      + \ltwo{\calL_h(\vhat_h)}
    \leq C \norm{(\vhat_h,v_h)}_{U_h} .
    $$
  \end{itemize}
\end{proof}

\begin{theorem}
  Let $\Omega \subset \Re^d$ be a bounded Lipschitz domain with boundary
  partition $\partial \Omega = \Gamma_0 \cup \Gamma_1$ with
  $|\Gamma_0| > 0$, and suppose that $\{\calT_h\}_{h>0}$ is a regular
  family of triangulations with each boundary partition triangulated by
  $\{\delTh\}_{h>0}$. Assume that $A \in \Linf^{d \times d}$ is
  elliptic, $f \in \Ltwo$, $u_0 \in \Hone$, and $g \in L^2(\Gamma_1)$,
  and let $u \in \Hone$ denote the solution of the elliptic equation
  \eqnref{:pde}.

  Then for each $h > 0$ there exists a unique solution
  $(u_h, \uhat_h, \bfp_h) \in \calP_0(\calT_h) \times \calP_0(\delTh)
  \times \calP_0(\calT_h)^d$ of the piecewise constant HDG
  \eqnref{:hdg0} scheme, and
  $$
  \lim_{h \rightarrow 0} \Big(
  \lranglenorm{u_h - \uhat_h}^{1/2}
  + \ltwo{u-u_h} + \ltwo{\nabla u - A^{-1} \bfp_h}
  \vph\Big) = 0.
  $$
\end{theorem}

\begin{proof}
  To accommodate the Dirichlet boundary data, let
  $(u_{0h}, \uhat_{0h}) = \iota_h(u_0)$ be the embedding of $u_0$
  defined in Lemma \ref{lem:CRdensity}. Since
  $\norm{(u_{0h}, \uhat_{0h})}_{U_h} \leq C \ltwo{\nabla u_0}$,
  the Lax Milgram Lemma shows that there exists a
  unique solution to the problem
  $(u_h, \uhat_h) - (u_{0h}, \uhat_{0h}) \in U_h$,
  $$
  \langle u_h-\uhat_h, v_h-\vhat_h \rangle
  + \left(A_h \nabla \calL_h(\uhat_h), \nabla \calL_h(\vhat_h) \vph\right)
  = (f,v_h) + \langle g, \vhat \rangle_{\Gamma_1},
  \qquad (v_h, \vhat_h) \in U_h,
  $$
  with
  $$
  \norm{(u_h, \uhat_h)}_{U_h}
  \leq C \left(\ltwo{f} + \norm{g}_{L^2(\Gamma_1)} + \hone{u_0} \vph\right).
  $$
  This problem is equivalent to the piecewise constant HDG scheme
  \eqnref{:hdg0} with $\bfp_h = A_h \nabla \calL_h(\uhat_h)$. The compactness
  properties of the Crouzeix Raviart space show that there exists
  $\ut \in \Hone$ and a subsequence such that
  $$
  \calL_h(\uhat_h) \rightarrow \ut \,\, \text{ in } \Ltwo, \quad
  \nabla \calL(\uhat)_h \weak \nabla \ut \,\, \text{ in } \Ltwo^d, \quad
  \calL_h(\uhat_h)|_{\partial \Omega} \rightarrow \ut|_{\partial \Omega}
  \,\, \text{ in } L^2(\partial\Omega);
  $$
  in addition, it follows from equation \eqnref{:hdg0p} 
  and lemma \ref{lem:consistent0} that
  $$
  \bfp_h \weak A \nabla \ut \equiv \bfp, 
  \qquad \text{ and } \qquad
  u_h \rightarrow \ut \,\, \text{ in } L^2(\Omega).
  $$
  Also, note that $\ut|_{\Gamma_0} = u_0|_{\Gamma_0}$ since
  $$
    \norm{\ut - u_0}_{L^2(\Gamma_0)}
    \leq \lim_{h\to 0} \Big(
     \norm{\ut - \calL_h(\uhat_h)}_{L^2(\Gamma_0)}
    +\norm{\calL_h(\uhat_h) - \uhat_h}_{L^2(\Gamma_0)}
    +\norm{\uhat_h - u_0}_{L^2(\Gamma_0)} \Big)
    = 0.
  $$
  This limit holds since
  Lemma \ref{lem:consistent0} shows the middle term vanishes, and
  the third term vanishes since the Dirichlet data is projected
  onto the boundary faces in equation \eqnref{:wpUo}. Next, fix
  $v \in U \equiv \{v \in \Hone \sst v|_{\Gamma_0} = 0\}$ and set
  $(v_h, \vhat_h) = \iota_h(v)$ in the weak statement \eqnref{:hdg0u}.
  Passage to the limit shows $\ut - u_0 \in U_h$ and
  $$
  (A \nabla \ut, \nabla v) = (f,v) + \langle g,v\rangle, \qquad v \in U,
  $$
  
  since $\lranglenorm{u_h - \uhat_h}$ is bounded and
  $\lranglenorm{v_h - \vhat_h} \rightarrow 0$.  Since solutions of
  the elliptic problem are unique, it follows that $\ut = u$, and it
  was unnecessary to pass to a subsequence in the passage to the limit.

  To establish strong convergence of the gradients, set
  $(v_h, \vhat_h, \bfq_h) = (u_h - u_{0h}, \uhat_h - \uhat_{0h},
  \bfp_h)$ in \eqnref{:hdg0} and take the limit to get
  $$
  \lim_{h\rightarrow 0} \Big(
    \lranglenorm{u_h - \uhat_h}
    + (A^{-1} \bfp_h, \bfp_h) - (\bfp, \nabla u_0) \Big)
  = (f,u-u_0) + \langle g,u-u_0 \rangle_{\Gamma_1}
  = (A^{-1} \bfp - \nabla u_0, \bfp).
  $$
  The second equality on the right follows upon setting the test
  function $v = u-u_0 \in U$ in the weak statement for $u$.  Since the map
  $\bfp \mapsto (A^{-1}\bfp, \bfp)^{1/2}$ is strictly convex,
  and convex functions are weakly lower
  semi--continuous, it follows that
  $ (A^{-1}\bfp, \bfp) \leq \liminf_{h \rightarrow 0} (A^{-1}\bfp_h,
  \bfp_h)$. Strong convergence of $\bfp_h$ to $\bfp$ in $\Ltwo^d$
  and convergence of $\lranglenorm{u_h - \uhat_h}$ to zero
  then follows. 
\end{proof}

\section{Stability and Convergence of the HDG Scheme} \label{sec:hdg}
In order to establish convergence of solutions of the HDG scheme
\eqnref{:wp}, a consistent approximation of the distributional
derivative is developed. This is introduced in the next section, and
essential properties required for the subsequent proof of stability
and convergence of the HDG scheme are presented.

\subsection{Discrete Distributional Gradient}
The discrete distributional gradient introduced next appears
intrinsically within the HDG scheme \eqnref{:wp}.

\begin{definition} \label{def:Gh} The lifting
  $\mbfG_h: \Honeh \times L^2(\delTh) \rightarrow \calPell(\calT_h)^d$ is the
  function characterized by
  $$
    (\mbfG_h(v,\vhat), \bfq_h)
    = (\nabla v , \bfq_h) 
    + \langle \vhat - v, \bfq_h \cdot\bfn  \rangle,
    \qquad \bfq_h \in \calPell(\calT_h)^d.
  $$
\end{definition}
This lifting is used to construct the natural analog of $\Hone$ in 
the HDG context.

\begin{definition} \label{def:Uh}
  Let $\calT_h$ be a triangulation of a bounded Lipschitz
  domain $\Omega \subset \Re^d$.
  \begin{itemize}
    \item Let $\Gamma_0 \subset \partial\Omega$ be the subset of the
    boundary where Dirichlet data is specified. Then $U_h$ is the
    subspace of $\calPell(\calT_h) \times \calPell(\delTh)$ with
    $\vhat_h$ vanishing on the $\Gamma_0$,
    $$
    U_h = \{(v_h, \vhat_h) \in \calPell(\calT_h) \times \calPell(\delTh),
    \sst \vhat_h|_{\Gamma_0} = 0\}.
    $$
  \item The norm $\norm{.}_{U_h}$ on $U_h$ is
    $$
    \norm{(u_h,\uhat_h)}^2_{U_h}
    = \lranglenorm{u_h - \uhat_h} + \ltwo{\bfG_h(u_h,\uhat_h)}^2.
    $$
    and is a semi--norm on $H^1(\calT_h) \times L^2(\delTh)$.

  \item The embedding
    $\iota_h: \Hone \rightarrow \calPell(\calT_h) \times
    \calPell(\delTh)$ is the function for which the components
    $(u_h, \uhat_h) = \iota_h(u)$ satisfy
    $$
    (u_h, v_h) = (u, v_h), \quad v_h \in \calPell(\calT_h),
    \quad \text{ and } \quad
    \langle \uhat_h, \vhat_h \rangle = \langle u, \vhat_h \rangle,
    \quad \vhat_h \in \calPell(\delTh).
    $$
    That is, $u_h$ is the $\Ltwo$ projection of $u$ onto
    $\calPell(\calT_h)$ and $\uhat_h$ is the $L^2(\delTh)$ projection
    of $u|_{\delTh}$ onto $\calPell(\delTh)$.
  \end{itemize}
  Note that $\mbfG_h(\iota_h(u))$ is the $\Ltwo$ projection of
  $\nabla u$ onto $\calPell(\calT_h)$.
\end{definition}

Consistency of the HDG scheme will follow from the convergence of the
embeddings $\iota_h(u)$.

\begin{lemma} \label{lem:GhDensity} Let $\calT_h$ be a triangulation
  of a domain $\Omega \subset \Re^d$ and
  $\iota_h: \Hone \rightarrow \calPell(\calT_h) \times \calPell(\delTh)$
  be the embedding characterized in Definition \ref{def:Uh}.  Then
  there exists a constant depending only upon the maximal aspect ratio
  of $\calT_h$ such that for all $u \in \Hone$ and
  $(u_h,\uhat_h) = \iota_h(u)$
   \begin{itemize}
    \item $\lranglenorm{u_h - \uhat_h} \leq C h \ltwo{\nabla u}^2$.
    \item $\ltwo{u_h - u}  \leq C h \ltwo{\nabla u}$.
    \item In addition, for $0 \leq s \leq \ell$, if $u \in H^{s+1}(\Omega)$ then
    $$
      \ltwo{u_h - u} + h \ltwo{\mbfG_h(u_h,\uhat_h) - \nabla u}
      \leq C h^{s+1} |u|_{H^{s+1}(\Omega)}.
      $$

    \item $\ltwo{\mbfG_h(u_h, \uhat_h)} \leq C \ltwo{\nabla u}$.

      In particular, for a regular family of meshes
      $\lim_{h \rightarrow 0}
      \ltwo{\mbfG_h(u_h, \uhat_h) - \nabla u} = 0$.
  \end{itemize}
\end{lemma}

These properties follow from standard parent element calculations, the
Bramble--Hilbert Lemma, and scaling \cite{Ci78,BrSc08}. The next lemma
establishes Poincar\'e and trace inequalities for the space
$(U_h,\norm{.}_{U_h})$ and compactness of bounded sequences.

\begin{lemma} \label{lem:HDGproperties} Let $\calT_h$ be a
  triangulation of a bounded Lipschitz domain $\Omega \subset \Re^d$,
  $\mbfG_h: \Honeh \times L^2(\delTh) \rightarrow \calPell(\calT_h)^d$
  be the lifting of Definition \ref{def:Gh}, and $\norm{.}_{U_h}$ and
  $U_h$ be the norm and subspace characterized in Definition
  \ref{def:Uh}.
  \begin{itemize}
  \item (Embedding) There exists a constant depending only upon the
    maximal aspect ratio of $\calT_h$ such that for all
    $(u_h,\uhat_h) \in \calPell(\calT_h) \times \calPell(\delTh)$
    $$
    \lp{u_h} \leq C \left( \ltwo{u_h} + \norm{(u_h, \uhat_h)}_{U_h} \right)
    \qquad 1 \leq p \leq \frac{2d}{{d-1}},
    $$
    and
    $$
    \norm{\uhat_h}_{L^2(\partial\Omega)}
    \leq C \left( \ltwo{u_h} + \norm{(u_h, \uhat_h)}_{U_h} \right).
    $$    
  \item (Poincar\'e) Let $\partial \Omega = \Gamma_0 \cup \Gamma_1$ be a
    partition of the boundary with $|\Gamma_0| > 0$, and suppose that
    each component is triangulated by $\delTh$. Then there exists a
    constant depending only upon the maximal aspect ratio of $\calT_h$
    such that
    $$
    \ltwo{u_h} \leq C \norm{(u_h,\uhat_h)}_{U_h},
    \qquad u_h \in U_h.
    $$

  \item (Compactness) Let $\{\calT_h\}_{h>0}$ be a regular family of
    triangulations of $\Omega$.  If
    $(u_h,\uhat_h) \in \calPell(\calT_h) \times \calP(\delTh)$ and
    $\{\ltwo{u_h} + \norm{(u_h,\uhat_h)}_{U_h}\}_{h > 0}$ is bounded, then
    there exists a subsequence (not relabeled) and $u \in \Hone$ such
    that as $h \rightarrow 0$
    $$
    u_h \rightarrow u \,\, \text{ in } \Ltwo, \quad
    \mbfG_h(u_h,\uhat_h) \rightharpoonup \nabla u \,\,
    \text{ in } \Ltwo^d, \quad
    \uhat_h|_{\partial \Omega} \rightharpoonup u|_{\partial \Omega}
    \,\, \text{ in } L^2(\partial\Omega).
    $$
  \end{itemize}
\end{lemma}

The proof of this lemma is given in the appendix.

\subsection{Convergence}
With $\mbfG_h: \Honeh \times L^2(\delTh) \rightarrow \calPell(\calT_h)^d$
denoting the lifting introduced in Definition \ref{def:Gh}, the HDG
scheme \eqnref{:wp} may be written as
$(u_h, \uhat_h, \bfp_h) \in \calPell(\calT_h) \times \calPell(\delTh)
\times \calPell(\calT_h)^d$
\begin{subequations}
\label{eqn:hdgG}
\begin{align}
  \label{eqn:hdg_pk_with_G_1}
  \langle u_h-\uhat_h, v_h-\vhat_h \rangle
  + \left(\bfp_h, \mbfG_h(v_h,\vhat_h) \vph\right)
  &= (f,v_h) + \langle g, \vhat \rangle_{\Gamma_1},
  \quad && (v_h, \vhat_h) \in U_h, \\
  \label{eqn:hdg_pk_with_G_2}
  \left(A^{-1} \bfp_h, \bfq_h \vph\right)
  - \left(\mbfG_h(u_h,\uhat_h), \bfq_h\vph\right)
  &= 0,
  \qquad&&  \bfq_h \in \calPell(\calT_h)^d, \\
  \label{eqn:hdg_pk_with_G_3}
  \langle \uhat_h, \vhat_h \rangle_{\Gamma_0}
  = \langle u_0, \vhat_h \rangle_{\Gamma_0}, &
  \qquad&& \vhat_h \in \calPell(\delTh \cap \Gamma_0).
\end{align}
\end{subequations}
where $U_h$ is the subspace of
$\calPell(\calT_h) \times \calPell(\delTh)$ introduced in Definition
\ref{def:Uh}.  Since $A$ is elliptic, the mapping
$A_h: \calPell(\calT_h) \rightarrow \calPell(\calT_h)$ characterized by
$$
\left(A^{-1} A_h(\bfg_h), \bfq_h\right)
= \left(\bfg_h, \bfq_h\right), 
\qquad \bfq_h\in \calPell(\calT_h)^d.
$$
is an isomorphism with continuity constants in $\Ltwo$ independent
of $h$.  Equation \eqnref{:hdg_pk_with_G_2} shows
$\bfp_h = A_h(\mbfG_h(u_h,\uhat_h))$, so equation
\eqnref{:hdg_pk_with_G_1} becomes
\begin{equation} \label{eqn:wpGh}
  \langle u_h-\uhat_h, v_h-\vhat_h \rangle
  + \left(A_h(\mbfG_h(u_h,\uhat_h)), \mbfG_h(v_h,\vhat_h) \vph\right)
  = (f,v_h) + \langle g, \vhat_h \rangle_{\Gamma_1},
\end{equation}
and the classical Lax Milgram Theorem can be used to establish
existence and stability of the HDG scheme.

\begin{theorem}
  Let $\Omega \subset \Re^d$ be a bounded Lipschitz domain with boundary
  partition $\partial \Omega = \Gamma_0 \cup \Gamma_1$ with
  $|\Gamma_0| > 0$, and suppose that $\{\calT_h\}_{h>0}$ is a regular
  family of triangulations with each boundary partition triangulated by
  $\{\delTh\}_{h>0}$. Assume that $A \in \Linf^{d \times d}$ is
  elliptic, $f \in \Ltwo$, $u_0 \in \Hone$, and $g \in L^2(\Gamma_1)$,
  and let $u \in \Hone$ denote the solution of the elliptic equation
  \eqnref{:pde}.

  Then for each $h > 0$ there exists a unique solution
  $(u_h, \uhat_h, \bfp_h) \in \calPell(\calT_h) \times \calPell(\delTh)
  \times \calPell(\calT_h)^d$ of the HDG scheme \eqnref{:wp}, and
  $$
  \lim_{h \rightarrow 0} \Big(
  \lranglenorm{u_h - \uhat_h}^{1/2}
  + \ltwo{u-u_h} + \ltwoh{\nabla u - A^{-1} \bfp_h}
  \vph\Big) = 0.
  $$
\end{theorem}

\begin{proof}
  To accommodate the Dirichlet boundary data, let
  $(u_{0h}, \uhat_{0h}) = \iota_h(u_0)$ be the embedding of $u_0$
  defined in Definition \ref{def:Uh}. Lemma \ref{lem:GhDensity} shows
  $\norm{(u_{0h}, \uhat_{0h})}_{U_h} \leq C \ltwo{\nabla u_0}$, so the
  Lax Milgram Lemma gives the existence of a unique solution to the
  problem $(u_h, \uhat_h) - (u_{0h}, \uhat_{0h}) \in U_h$,
  $$
  \langle u_h-\uhat_h, v_h-\vhat_h \rangle
  + \left(
    A_h(\mbfG_h(u_h,\uhat_h)), \mbfG_h(v_h,\vhat_h) \vph\right)
  = (f,v_h) + \langle g, \vhat \rangle_{\Gamma_1},
  \quad (v_h, \vhat_h) \in U_h,
  $$
  with
  $$
  \norm{(u_h, \uhat_h)}_{U_h}
  \leq C \left(\ltwo{f} + \norm{g}_{L^2(\Gamma_1)} + \hone{u_0} \vph\right).
  $$
  This problem is equivalent to the HDG scheme \eqnref{:wp} with
  $\bfp_h = A_h(\mbfG_h(u_h,\uhat_h))$. When $A$ is elliptic,
  the compactness properties
  of Lemma \ref{lem:HDGproperties} show that there exists
  $\ut \in \Hone$ and $\bfp \in \Ltwo^d$ and a subsequence such that
  \begin{align*}
    \mbfG_h(u_h, \uhat_h) &\weak \nabla \ut &&\,\, \text{ in } \Ltwo^d, \qquad
    && \bfp_h \weak \bfp &\,\, \text{ in } \Ltwo^d, \\\
    u_h &\rightarrow \ut &&\,\, \text{ in } \Ltwo, \qquad
    && \uhat_h|_{\partial \Omega} \weak \ut|_{\partial \Omega}
    &\,\, \text{ in } L^2(\partial\Omega);
  \end{align*}
  in addition, it follows from equation \eqnref{:hdg_pk_with_G_2} that
  $A^{-1} \bfp_h \weak \nabla \ut$, so $\bfp = A \nabla \ut$.

  To establish consistency, first note that
  $\ut|_{\Gamma_0} = u_0|_{\Gamma_0}$ since equation
  \eqnref{:hdg_pk_with_G_3} shows $\uhat_h|_{\Gamma_0}$ is the
  projection of $u_0$ onto $\calPell(\Gamma_0)$. Next, fix
  $v \in U \equiv \{v \in \Hone \sst v|_{\Gamma_0} = 0\}$ and set
  $(v_h, \vhat_h) = \iota_h(v)$ in the weak statement
  \eqnref{:hdg_pk_with_G_1}. Passage to the limit shows
  $\ut - u_0 \in U_h$ and
  $$
  (A \nabla \ut, \nabla v) = (f,v) + (g,v), \qquad v \in U,
  $$
  since $\lranglenorm{u_h - \uhat_h}$ is bounded and
  $\lranglenorm{v_h - \vhat_h} \rightarrow 0$.  Since solutions of
  the elliptic problem are unique, it follows that $\ut = u$, and it
  was unnecessary to pass to a subsequence in the passage to the limit.

  To establish strong convergence of the gradients, set
  $(v_h, \vhat_h, \bfq_h) = (u_h - u_{0h}, \uhat_h - \uhat_{0h},
  \bfp_h)$ in \eqnref{:hdg_pk_with_G_1} and take the limit to get
  $$
  \lim_{h\rightarrow 0} \Big(
    \lranglenorm{u_h - \uhat_h}
    + (A^{-1} \bfp_h, \bfp_h) - (\bfp, \nabla u_0) \Big)
  = (f,u-u_0) + \langle g,u-u_0 \rangle_{\Gamma_1}
  = (A^{-1} \bfp - \nabla u_0, \bfp).
  $$
  The second inequality on the right follows upon setting the test
  function $v = u-u_0 \in U$ in the weak statement for $u$.  Since the map
  $\bfp \mapsto (A^{-1}\bfp, \bfp)^{1/2}$ is strictly convex,
  and convex functions are weakly lower
  semi--continuous, it follows that
  $ (A^{-1}\bfp, \bfp) \leq \liminf_{h \rightarrow 0} (A^{-1}\bfp_h,
  \bfp_h)$. Strong convergence of $\bfp_h$ to $\bfp$ in $\Ltwo^d$
  and convergence of $\lranglenorm{u_h - \uhat_h} $ to zero
  then follows. 
\end{proof}

\section{Variational Problems} \label{sec:hdv}
This section establishes the stability and convergence of HDG 
formulations of variational problems which seek minima of
functions $I:\Hone \rightarrow \Re$ of the form
\begin{equation} \label{eqn:vp}
I(u) = \int_\Omega f(x, u(x), \nabla u(x)) \, dx
+ \int_{\Gamma_1} g(s, u(s)) \, da(s)
\qquad u - u_0 \in U,
\end{equation}
where $U = \{u \in \Hone \sst u|_{\Gamma_0} = 0\}$. Here
$\Omega \subset \Re^d$ is a bounded domain and
$\partial \Omega = \Gamma_0 \cup \Gamma_1$ a partition of
the boundary. Adopting the notation of Section \ref{sec:hdg}, the
HDG formulation of the variational problem seeks minima
of $I_h : \calPell(\calT_h) \times \calPell(\delTh) \rightarrow \Re$,
\begin{equation} \label{eqn:vph}
I_h(u_h,\uhat_h) = \half \lranglenorm{u_h - \uhat_h}
+ \int_\Omega f(., u_h, \mbfG_h(u_h,\uhat_h))
+ \int_{\Gamma_1} g \, \uhat_h,
\end{equation}
with $(u_h, \uhat_h) - \iota_h(u_0) \in U_h$; that is,
$$
(u_h, \uhat_h)
\in \Big\{(u_h, \uhat_h) \in \calPell(\calT_h) \times \calPell(\delTh)
\sst \int_k \uhat_h \vhat_h = \int_k u_0 \vhat_h, \,\,
\vhat_h \in \calPell(k), \,\, k \subset \Gamma_0 \Big\}.
$$

\begin{example} If $A \in \Linf^{d \times d}$ is symmetric and
  $f(x,u,\bfp) = (1/2) |\bfp|^2_A - f u$, the Euler Lagrange equation
  of \eqnref{:vph} is \eqnref{:wpGh} with $A_h = A$.
\end{example}

The following conditions on the integrands will ensure measurability
of the integrands.

\begin{definition}[Caratheodory Function] Let $\Omega \subset \Re^d$
  be a domain, then
  $f:\Omega \times \Re^m \times \Re^M \rightarrow \Re \cup \{\infty\}$
  is {\em Caratheodory} if
  \begin{itemize}
  \item $x \mapsto f(x,u,\xi)$ is measurable for every
    $(u,\xi) \in \Re^m \times \Re^M$.
    
  \item $(u,\xi) \mapsto f(x, u,\xi)$ is continuous for
    almost every $x \in \Omega$.
  \end{itemize}
\end{definition}

The following two theorems establish continuity and weak lower
semi--continuity of variational problems with Caratheodory integrands.

\begin{lemma}[Nemytskii Operators \cite{Sho97}]
  \label{lem:Nemytskii}
  Let $\Omega \subset \Re^d$ be open and $f:\Omega \times
  \Re^m \times \Re^M \rightarrow \Re$ be Caratheodory. Let
  $1 \leq p,q,r < \infty$, $k \in \Lr$, $C \geq 0$, and assume that
  $$
  |f(x,u,\xi)| \leq C \left(|u|^{p/r} + |\xi|^{q/r} \vph\right) + k(x).
  $$
  Then the operator $F:\Lp \times \Lq \rightarrow \Lr$ defined by
  $$
  F(u,\xi)(x) = f(x, u(x), \xi(x)), \qquad a.e \,\, x \in \Omega,
  $$
  is bounded and strongly continuous.
\end{lemma}

\begin{theorem} [Weak Lower Semi--Continuity
  \mbox{\cite[Theorem 3.23]{DaBe08}}] \label{thm:lsc}
  Let $\Omega \subset \Re^d$ be open and
  $p, q \geq 1$. Let $f : \Omega \times \Re^m \times \Re^M \rightarrow
  \Re \times \{\infty\}$ be a Caratheodory function satisfying
  $$
  f (x, u, \xi) \geq (a (x). \xi) + b (x) + c |u|^p
  $$
  for almost every $x \in \Omega$, for every
  $(u, \xi) \in \Re^m \times \Re^M$, for some $a \in L^{q'}(\Omega)^M$,
  $1/q + 1/q' = 1$, $b \in \Lone$, and $c \in \Re$. Let
  $$
  I(u, \xi) = \int_\Omega f (x, u (x) , \xi (x)) dx.
  $$
  Assume that $\xi \rightarrow f (x, u, \xi)$ is convex and that
  $$
    \lim_{n \rightarrow \infty} u_n \rightarrow u \, \text{ in } \Lp^m,
  \qquad \text{ and } \qquad
  \lim_{n \rightarrow \infty} \xi_n \weak \xi \, \text{ in } \Lq^M.
  $$
  Then
  $
  \liminf_{n \rightarrow \infty} I(u_n , \xi_n ) \geq I(u, \xi) .
  $
\end{theorem}

The following theorem establishes convergence minima of the discrete
variational problem. To simplify the exposition, the lower order terms
are assumed to have at most quadratic growth. The embedding results of
Lemma \ref{lem:CRproperties} and Lemma \ref{lem:HDGproperties}.  can
be used to generalize this.

\begin{theorem}
  Let $\Omega \subset \Re^d$ be a bounded Lipschitz domain with boundary
  partition $\partial \Omega = \Gamma_0 \cup \Gamma_1$ with
  $|\Gamma_0| > 0$, and suppose that $\{\calT_h\}_{h>0}$ is a regular
  family of triangulations with each boundary partition triangulated by
  $\{\delTh\}_{h>0}$. Assume that $f:\Omega \times \Re \times \Re^d
  \rightarrow \Re$ is Caratheodory, is convex in its
  third argument, and satisfies
  $$
  a_0(x) + a_1(x) |u| + a_2 |\bfg|^2
  \leq f(x,u,\bfg)
  \leq b_0(x) + b_1 |u|^2 + b_2 |\bfg|^2,
  $$
  for some $a_0, b_0 \in \Lone$, $a_1 \in \Ltwo$ and
  $a_2, b_1, b_2 > 0$. Assume that $g \in L^2(\Gamma_1)$ and $u_0 \in
  \Hone$.

  Then for each $h > 0$ there exist minimizers
  $(u_h, \uhat_h) \in \iota_h(u_0) + U_h$ of the HDG variational
  problem \eqnref{:vph}. If $\{(u_h, \uhat_h)\}_{h>0}$ are minimizers
  of \eqnref{:vph}, then there exists a subsequence and minimizer
  $u \in \Hone$ of \eqnref{:vp} such that
  $$
  I_h(u_h,\uhat_h) \rightarrow I(u), \quad
  u_h \rightarrow u, \, \text{ in } \Ltwo, \quad
  \mbfG_h(u_h,\uhat_h) \weak \nabla u
  \, \text{ in } \Ltwo^d, \quad
   \lranglenorm{u_h - \uhat_h} \rightarrow 0.
   $$
   In addition, if $f(.,.,.)$ is strictly convex in its third argument
   then $\mbfG_h(u_h,\uhat_h) \rightarrow \nabla u$ in $\Ltwo^d$ and
   the whole sequence converges.
\end{theorem}

\begin{proof}
  The hypotheses guarantee that the variational problem \eqnref{:vp}
  has minimizers in $U$, and if $f(.,.,.)$ is strictly convex in
  its third argument then the minimizer is unique \cite{DaBe08}.
  
  If $u \in \Hone$, setting $(u_h, \uhat_h) = \iota_h(u)$ gives
  $$
  I_h(u_h, \uhat_h)
  = \half \lranglenorm{u_h - \uhat_h}
  + \int_\Omega f(., u_h, \mbfG_h(u_h, \uhat_h))
  + \int_{\Gamma_1} g(., \uhat_h).
  $$
  The upper bounds on the integrands and strong convergence of $u_h$,
  $\mbfG_h(u_h, \uhat_h)$, and $\uhat_h|_{\Gamma_1}$ guaranteed by Lemma
  \ref{lem:GhDensity} establish the hypotheses of the Nemytskii Lemma
  \ref{lem:Nemytskii} so that
  $$
  \lim_{h \rightarrow 0} I_h(u_h, \uhat_h)
  = \int_\Omega f(.,u,\nabla u) + \int_{\Gamma_1} g \,u
  = I(u).
  $$
  It follows that
  \begin{equation} \label{eqn:limInf}
    \lim_{h\rightarrow 0} \inf_{(u_h, \uhat_h) \in U_h} I_h(u_h, \uhat_h)
    \leq \inf_{u \in U} I(u).
  \end{equation}
  
  For functions $(u_h, \uhat_h) \in U_h$, the lower bounds on
  (coercivity of) the integrands show
  \begin{eqnarray*}
    I_h(u_h, \uhat_h)
    &\geq& \half \lranglenorm{u_h - \uhat_h}
           - \left( \lone{a_0} + \norm{\alpha_0}_{L^1(\Gamma_1)} \vph\right) \\
    && - \left( \ltwo{a_1} \ltwo{u_h}
    + \norm{\alpha_1}_{L^2(\Gamma_1)} \norm{\uhat_h}_{L^2(\Gamma_1)}
       \vph\right)
       + a_2 \ltwo{\mbfG_h(u_h,\uhat_h)}^2 \\
    &\geq& \half \lranglenorm{u_h - \uhat_h}
           - C(\epsilon)
           - \epsilon \left( \ltwo{u_h}^2 + \norm{\uhat_h}^2_{L^2(\Gamma_1)}
           \vph\right)
           + a_2 \ltwo{\mbfG_h(u_h,\uhat_h)}^2 \\
    &\geq& \half \min(1/2,a_2) \norm{(u_h,\uhat_h)}^2_{U_h} - C, 
  \end{eqnarray*}
  where the last line follows from the trace and Poincar\'e inequality
  of Lemma \ref{lem:HDGproperties} and an appropriate choice of
  $\epsilon > 0$. Since $U_h$ is finite dimensional it is immediate
  that minimizers of $I_h: U_h \rightarrow \Re$ exist, since it is
  continuous and coercive, and bounded sets of $U_h$ are compact.

  Letting $\{(u_h,\uhat_h)\}_{h > 0}$ be minimizers of $I_h(.)$ on
  $U_h$, it follows that $\{\norm{(u_h,\uhat_h)}_{U_h}\}_{h>0}$ is
  bounded, and the compactness properties of Lemma
  \ref{lem:HDGproperties} then show that there exists $\ut \in \Hone$
  and a subsequence with
  $$
  u_h \rightarrow \ut \,\, \text{ in } \Ltwo, \quad
  \mbfG_h(u_h,\uhat_h) \rightharpoonup \nabla \ut \,\,
  \text{ in } \Ltwo^d, \quad
  \uhat_h|_{\Gamma_1} \weak \ut|_{\Gamma_1}
  \,\, \text{ in } L^2(\partial\Gamma_1).
  $$
  It follows that $\ut \in U$, and (along the subsequence)
  \begin{eqnarray*}
    \lim_{h \rightarrow 0} I_h(u_h, \uhat_h) 
    &=& \lim_{h \rightarrow 0} \left(
        \half \lranglenorm{u_h - \uhat_h}
        + \int_\Omega f(.,u_h, \mbfG_h(u_h,\uhat_h))
        + \int_{\Gamma_1} g \, \uhat_h \right) \\
    &\geq& \liminf_{h \rightarrow 0}
           \half \lranglenorm{u_h - \uhat_h}
           + \int_\Omega f(.,\ut, \nabla \ut)
           + \int_{\Gamma_1} g \, \ut \\
    &\geq& \liminf_{h \rightarrow 0}
           \half \lranglenorm{u_h - \uhat_h}
           + I(\ut),
  \end{eqnarray*}
  where the middle line follows from the lower semi--continuity
  guaranteed by Theorem \ref{thm:lsc} and continuity of the boundary
  integral under weak convergence. Combining this estimate with the
  lower bound \eqnref{:limInf} gives $\ut \in U$, and
  $$
  \liminf_{h \rightarrow 0}
  \half \lranglenorm{u_h - \uhat_h}
  + I(\ut)
  \leq \lim_{h\rightarrow 0} I_h(u_h, \uhat_h)
  \leq \inf_{u \in U} I(u).
  $$
  It follows that
  $ I_h(u_h, \uhat_h) \rightarrow I(\ut) = \inf_{u \in U} I(u)$ and
  $\lranglenorm{u_h - \uhat_h} \rightarrow 0$.
\end{proof}

\begin{figure}
  \hspace*{-0.375in}
  \begin{tabular}{|l|cc|cc|cc|} \hline
    & \multicolumn{2}{c|}{\text{weight $\tau = 1$}}
    & \multicolumn{2}{c|}{\text{weight $\tau = h^{-1}$}}
    & \multicolumn{2}{c|}{\text{weight $\tau = h$}}
    \\ \hline
    \multicolumn{1}{|c|}{$h$}
    & $\ltwoh{u-u_h}$  & $\ltwoh{\bfp-\bfp_h}$
    & $\ltwoh{u-u_h}$  & $\ltwoh{\bfp-\bfp_h}$ 
    & $\ltwoh{u-u_h}$  & $\ltwoh{\bfp-\bfp_h}$
    \\ \hline
    1/4 & 1.641155 & 0.436908 & 2.784128 & 0.685960 & 0.651680 & 0.246280 \\
    1/8 & 1.004204 & 0.276075 & 2.743754 & 0.658023 & 0.336693 & 0.147604 \\
    1/16 & 0.571905 & 0.161664 & 2.685806 & 0.632281 & 0.365050 & 0.094772 \\
    1/32 & 0.317957 & 0.089354 & 2.715682 & 0.636153 & 0.416482 & 0.050470 \\
    1/64 & 0.165965 & 0.050218 & 2.719005 & 0.635747 & 0.452205 & 0.032064 \\
    1/128 & 0.084683 & 0.029143 & 2.722563 & 0.636480 & 0.480743 & 0.021351 \\
    \hline
  \end{tabular}
  \caption{Piecewise constant HDG approximation of solution to
    equation \eqnref{:NotSmooth}.}
\label{fig:hdg00}
\end{figure}

\begin{figure}
  \begin{center}
    \includegraphics[width=2.5in]{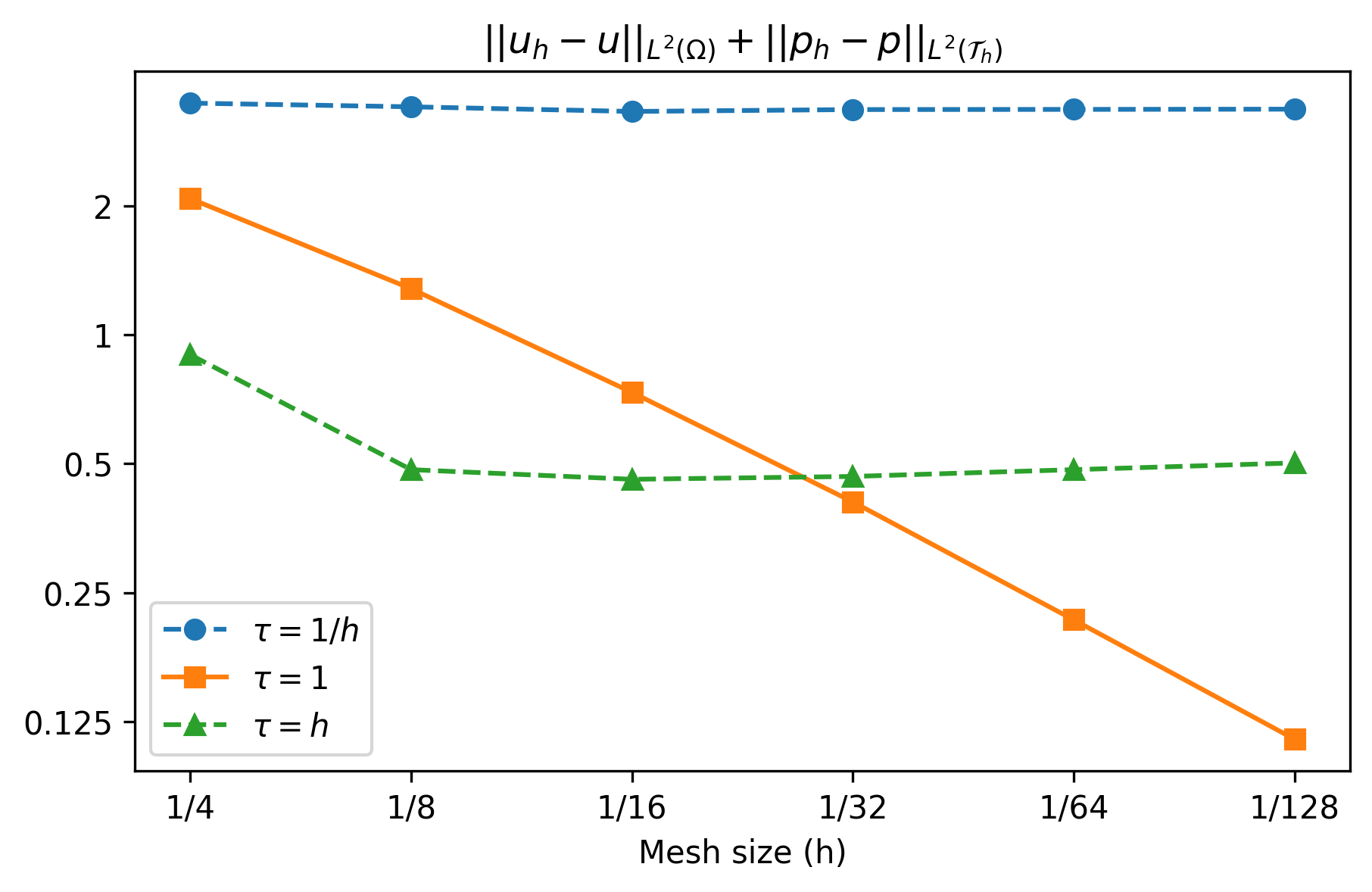}
    \qquad \qquad
    \includegraphics[width=2.5in]{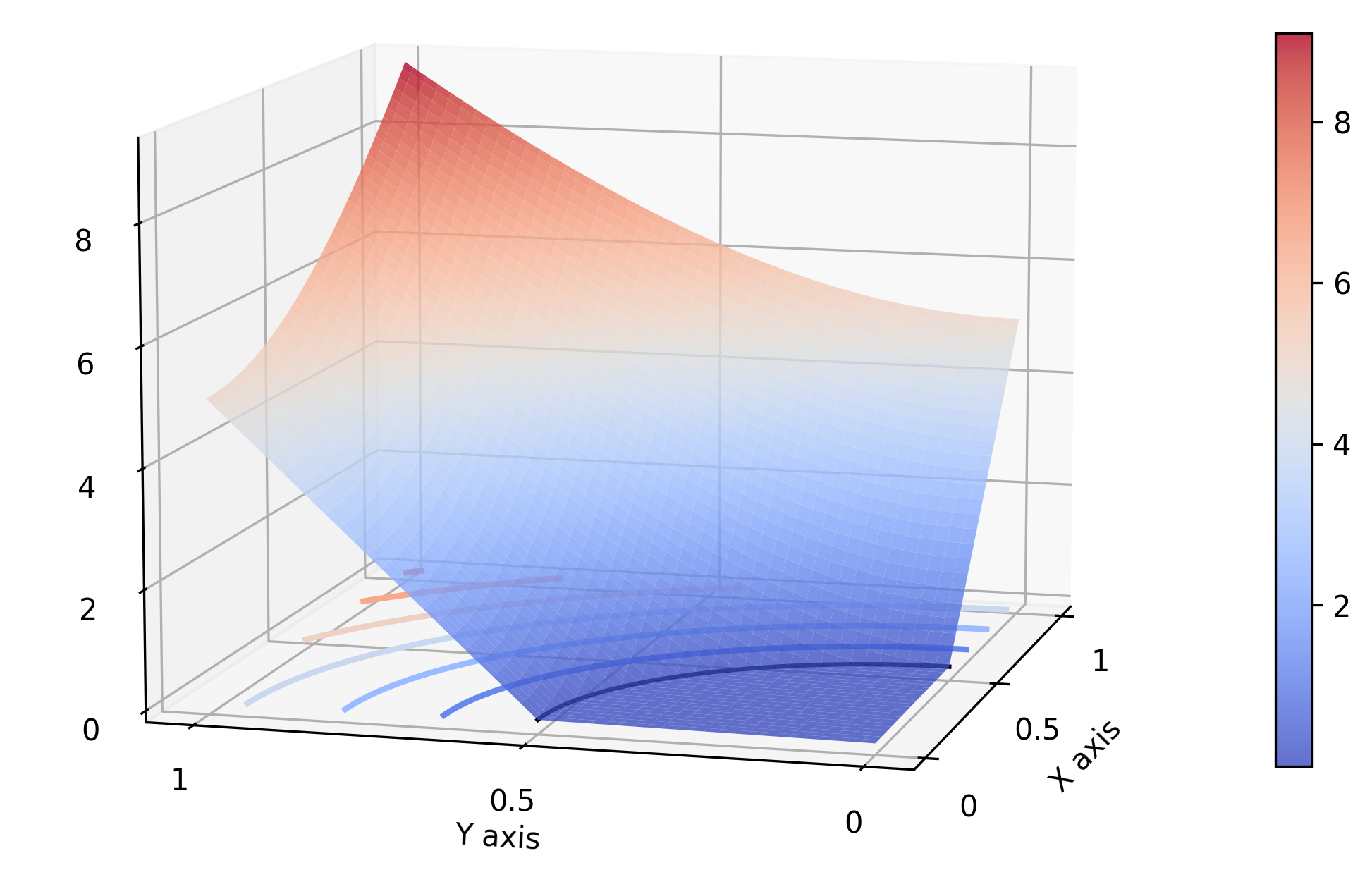}
  \end{center}
  \caption{$L^{2}(\calT_h)$ errors for different weights (left) and
   the exact solution (right).} \label{fig:egRates}
\end{figure}

\section{Numerical Example}
This section illustrates the convergence of the HDG scheme for a
problem with non--smooth solution. The piecewise smooth function
$$
u(x) = 
\begin{cases} 
  \alpha |x|, & |x| \leq 1/2, \\
  \alpha/2 + \beta (|x| - 1/2), & |x| \geq 1/2,
\end{cases}
$$
is a solution of
\begin{equation} \label{eqn:NotSmooth}
-\rmdiv(A \nabla u) = \frac{- \alpha \beta}{|x|} \quad \text{ in } (0,1)^2,
\quad \text{ with } \quad
A = 
\begin{cases} 
  \beta I, & |x| \leq 1/2, \\
  \alpha  I, & |x| > 1/2.
\end{cases}
\end{equation}
For the numerical experiments, Dirchlet data is specfied on the top
and right hand sides of the square and Neumann data on the bottom and
left, and we select $\alpha = 1/10$ and $\beta = 10$.

Figures \ref{fig:hdg00} tabulates the errors for solutions computed
using the piecewise constant HDG schemes on uniform meshes with
weights $\tau = 1$, $1/h$, and $h$. Figure \ref{fig:egRates}  plots the
exact solution and the dependence of the errors on mesh size. As in
Example \ref{eg:smooth}, $u_h$ and $\bfp_h$ both converge only when
the unit weight is utilized.

\appendix

\section{Proofs of Lemmas}
The following lemma relating integrals over $k \in \delTh$ and
$K \in \calT_h$ is used ubiquitously in the analysis of DG schemes.

\begin{lemma} \label{lem:scaleK}
  Let $K \subset \Re^d$ be a simplex with inscribing ball
  $B(x_0,\rho_K)$ and set $r_K = \max_{x \in K} |x-x_0|$ .
\begin{enumerate}
\item If $u \in H^1(K)$ then
  $$
  \rho_K \int_{\partial K} u^2
  = \int_K \left(d \, u^2 + (x-x_0).\nabla u^2 \vph\right)
  \leq \int_K \left( (d+1)u^2 + r_K^2 |\nabla u|^2 \vph\right).
  $$

\item If $u_h \in \calPell(\calT_h)$ then exists a constant $C > 0$
  depending only upon $(r_K/\rho_K)$ and $k$ such that
  $$
  \norm{u_h}_{L^2(\partial K)}
  \leq \left(C / \sqrt{\rho_K} \right) \norm{u_h}_{L^2(K)}.
  $$
\end{enumerate}
\end{lemma}

\subsection{Crouzeix Raviart Space}
Various trace and compactness properties of functions in the Crouzeix
Raviart space can be found in
\cite{Br03,ErGu04,TemamNavierStokes}. The ideas
introduced in Brenner \cite{Br03} facilitate a unified
proof of these results.

\begin{proof}[\textbf{Proof of Lemma \ref{lem:CRproperties}}]
  Brenner \cite{Br03} constructed operators
  $$
  E:CR(\calT_h) \rightarrow \calP_d(\calT_h) \cap \Hone
  \quad \text{ and } \quad
  F: \calP_d(\calT_h) \cap \Hone \rightarrow CR(\calT_h)
  $$
  satisfying 
  $F(E(u_h)) = u_h$, and 
  \begin{align*}
  \ltwo{E(u_h)} \leq C \ltwo{u_h}, \qquad
  \ltwo{\nabla E(u_h)} \leq C \ltwoh{\nabla u_h},  \qquad
  &u_h \in CR(\calT_h), \\
  \ltwo{E(u_h) - u_h} 
  + \sqrt{h} \norm{E(u_h) - u_h}_{L^2(\partial \Omega)} 
  \leq C h \ltwoh{\nabla u_h},
  \qquad & u_h \in CR(\calT_h).
  \end{align*}
  and if $w_h \in \calP_d(\calT_h) \cap \Hone$,
  $\ltwoh{\nabla F(w_h)} \leq C \ltwo{\nabla w_h}$,
  \begin{align*}
    \lp{F(w_h)} &\leq C \lp{w_h}, && 1 \leq p \leq \frac{2d}{d-2}, \\
    \norm{F(w_h)}_{L^p(\partial \Omega)}
    &\leq C \norm{w_h}_{L^p(\partial \Omega)}, 
    && 1 \leq p \leq \frac{2(d-1)}{d-2}.
  \end{align*}
  These last two inequalities only stated for $p = 2$ in
  \cite{Br03}; however, their proofs follow
  directly from a parent element calculation. In addition
  $$
  E(u_h) = F(u_h) = u_h, \qquad u_h \in \calP_1(\calT_h) \cap \Hone.
  $$
  
  \begin{enumerate}
  \item (Embedding) Letting $1 \leq p \leq 2 d / (d-2)$ ($p < \infty$
    if $d=2$), it follows from the Sobolev embedding theorem that
    for all $u_h \in CR(\calT_h)$,
    $$
    \lp{u_h} = \lp{F(E(u_h))}
    \leq C \lp{E(u_h)}
    \leq C \hone{E(u_h)}
    \leq C \norm{u_h}_{H^1(\calT_h)}.
    $$
    Similarly, if $1 \leq p \leq 2 (d-1)/(d-2)$, 
    $$
    \norm{u_h}_{L^p(\partial \Omega)}
    = \norm{F(E(u_h))}_{L^p(\partial \Omega)}
    \leq C \norm{E(u_h)}_{L^p(\partial \Omega)}
    \leq C \hone{E(u_h)}
    \leq C \norm{u_h}_{H^1(\calT_h)}.
    $$

  \item (Poincar\'e) Brenner \cite{Br03} proved the
    Poincar\'e inequality for a broad class of DG spaces which
    included subspaces of $CR(\calT_h)$ with averages vanishing
    on subsets $\Gamma_0 \subset \partial \Omega$.

  \item (Compactness) If $\{u_h\}_{h > 0} \subset CR(\calT_h)$
    and $\{\norm{u_h}_{H^1(\calT_h)}\}_{h > 0}$ is bounded, then
    so too is $\{\hone{E(u_h)}\}_{h > 0}$. It follows that
    upon passing to a subsequence there exists $u \in \Hone$
    with $E(u_h) \weak u$ in $\Hone$. Since $\Hone$
    is compactly embedded in $\Ltwo$ it follows that
    $\ltwo{E(u_h) - u} \rightarrow 0$ and
    \begin{eqnarray*}
      \ltwo{u_h - u}
      &\leq& \ltwo{u_h - E(u_h)} + \ltwo{E(u_h) - u} \\
      &\leq& C \ltwoh{\nabla u_h} h + \ltwo{E(u_h) - u} \\
      &\rightarrow& 0. \\
      \norm{u_h - u}_{L^2(\partial \Omega)}
      &\leq&
      \norm{u_h - E(u_h)}_{L^2(\partial \Omega)}
      + \norm{E(u_h) - u}_{L^2(\partial \Omega)}\\
      &\leq&
      C \ltwoh{\nabla u_h} \sqrt{h}
      + \norm{E(u_h) - u}_{L^2(\partial \Omega)}\\
      &\rightarrow& 0.
    \end{eqnarray*}
    To verify that (the broken) gradients \{$\nabla u_h\}_{h>0}$
    converge weakly to $\nabla u$, note first that they are
    bounded in $\Ltwo$ so, upon passage to a subsequence, have a weak
    limit, $\bfg \in \Ltwo$. To verify that $\bfg = \nabla u$, let
    $\bfq \in C_c^\infty(\Omega)^d$ and $\bfq_h \in RT_0(\calT_h)$ be
    the Raviart Thomas interpolant of $\bfq$. Then
    \begin{eqnarray*}
      (\nabla u_h,\bfq_h)
      &=& - (u_h, \rmdiv(\bfq_h))
          + \langle u_h, \bfq_h \cdot \bfn \rangle \\
      &=& - (u_h, \rmdiv(\bfq_h))
          + \sum_{k \in \delTh \cap \Omega} \int_k [u_h] \, \bfq_h. \bfn
          + \int_{\partial\Omega} u_h \, \bfq_h.\bfn \\
      &=& - (u_h, \rmdiv(\bfq_h));
    \end{eqnarray*}
    the last line following since $\bfq_h.\bfn$ is constant on the
    faces $k \in \delTh$ and vanishes on $\partial \Omega$. Since
    $\bfq_h \rightarrow \bfq$ in $\Hdiv$ and $u_h \rightarrow u$ in
    $\Ltwo$ it follows that $\nabla u_h \weak \nabla u$.
  \end{enumerate}
\end{proof}

\subsection{Discrete Distributional Derivative}

The lifting
$\mbfG_h: \Honeh \times L^2(\delTh) \rightarrow \calPell(\calT_h)^d$
introduced in Definition \ref{def:Gh} is local in the sense that
$\mbfG_h|_K$ is determined by $v|_K$ and $\vhat|_{\partial K}$ for
each $K \in \calT_h$. In order to establish global quantities, such as
the Poincar\'e inequality, it is convenient to couple $\mbfG_h$ with the
lifting into the Crouzeix Raviart space introduced in Definition \ref{def:Lh}.

\begin{lemma} \label{lem:Ghproperties}
  Let $\calT_h$ be a triangulation of a bounded Lipschtiz domain
  $\Omega \subset \Re^d$.
  \begin{enumerate}
  \item If $(v_h, \vhat_h) \in \calPell(\calT_h) \times L^2(\delTh)$,
    then $\nabla \calL_h(\vhat_h)$ is the projection of
    $\mbfG_h(v_h,\vhat_h)$ onto $\calP_0(\calT_h)^d$ in $\Ltwo^d$.
    In particular,
    $\norm{\nabla \calL_h(\vhat_h)}^2_{L^2(\calT_h)} \leq
    \ltwo{\mbfG_h(v_h,\vhat_h)}^2$, and if $\ell = 0$ then
    $ \nabla \calL_h(\vhat_h) = \mbfG_h(v_h,\vhat_h)$.
    
  \item If
    $(v_h,\vhat_h) \in \calPell(K) \times \calPell(\partial K)$, 
    then there exists $C > 0$ depending only upon the aspect ratio
    of $K$ such that
    $$
    h_K \norm{\nabla v_h}_{L^2(K)}^2 \leq C
    \left(h_K \norm{\mbfG_h(v_h,\vhat_h)}_{L^2(K)}^2
    + \lranglenormElement{\vhat_h - v_h}^2
    \right).
    $$
  
  \item If
    $(v_h,\vhat_h) \in \calPell(K) \times \calPell(\partial K)$, 
    then there exists $C > 0$ depending only upon the aspect ratio
    of $K$ such that
    $$
    \norm{v_h - \calL_h(\vhat_h)}_{L^2(K)}
    \le C \left( h_K \norm{\mbfG_h(v_h,\vhat_h)}_{L^2(K)}
      + \sqrt{h_K} \lranglenormElement{\vhat_h - v_h} \vph\right) .
    $$
    
  \item If $(v_h, \vhat_h) \in \calPell(\calT_h) \times L^2(\delTh)$, then
    there exists $C > 0$ depending only upon the aspect ratio of
    $\calT_h$ such that
    $$
    \honeh{\calL_h(\vhat_h)} \leq C \left(
      \ltwo{\mbfG_h(v_h,\vhat_h)}
      + \sqrt{h} \lranglenorm{v_h - \vhat_h}^{1/2}
      + \ltwo{v_h} \vph\right).
    $$
  \end{enumerate}
\end{lemma}

\begin{proof}
  \begin{enumerate}
  \item If $\bfq_h \in \calP_0(\calT_h)^d$ then the broken divergence
    vanishes, $\rmdiv(\bfq_h) = 0$, whence
    \begin{eqnarray*}
      (\mbfG_h(v_h, \vhat), \bfq_h)
      &\equiv& (\nabla v_h, \bfq_h)
               + \langle \vhat_h - v_h, \bfq_h.\bfn \rangle \\
      &=&  \langle \vhat_h, \bfq_h.\bfn \rangle
          = \langle \calL_h(\vhat_h), \bfq_h.\bfn \rangle
          = (\nabla \calL_h(\vhat_h), \bfq_h).
    \end{eqnarray*}
    
  \item Fixing $K \in \calT_h$ and selecting the test function
    in the definition of $\mbfG_h(v_h,\vhat_h)$ to be
    $\bfq_h = \mbfG_h(v_h, \vhat_h) - \nabla v_h$ on $K$ and zero
    otherwise gives
  \begin{eqnarray*}
    \norm{\mbfG_h(v_h, \vhat_h) - \nabla v_h}^2_{L^2(K)}
    &=&(\mbfG_h(v_h, \vhat_h) - \nabla v_h,
        \mbfG_h(v_h, \vhat_h) - \nabla v_h)_{L^2(K)} \\
    &=&\int_{\partial K}(\vhat_h - v_h)(\mbfG_h(v_h, \vhat_h)
        - \nabla v_h).\bfn\\
    &\leq& 
    \lranglenormElement{\mbfG_h(v_h, \vhat_h) - \nabla v_h}\,
    \lranglenormElement{\vhat_h - v_h}
    \\
    &\leq& (C / \sqrt{h_K})
    \norm{\mbfG_h(v_h, \vhat_h) - \nabla v_h}_{L^2(K)}
    \lranglenormElement{\vhat_h - v_h},
  \end{eqnarray*}
  where the last line follows from Lemma \ref{lem:scaleK}.  It follows
  that
  $$
  h_K \norm{\mbfG_h(v_h, \vhat_h) - \nabla v_h}^2_{L^2(K)} \leq C
    \lranglenormElement{\vhat_h - v_h}^2,
  $$
  and the stated bound follows from the triangle inequality.
  
\item If $u_h \in H^1(K)$ let $u_\partial \in \calP_0(\calT_h)$ be the
    function taking the average over the boundary of each element,
    $$
    u_\partial|_K = \frac{1}{|\partial K|} \int_{\partial K} u_h,
    \qquad \text{ in particular, } \qquad
    (\calL_h(\vhat_h) - v)_\partial
    = \frac{1}{|\partial K|} \int_{\partial K} (\vhat_h-v).    
    $$
    An application of Hölder's inequality shows
    $$
    \norm{\left(\calL_h(\vhat_h) - v_h\right)_\partial}^2_{L^2(K)} 
    \leq \frac{|K|}{|\partial K|} \lranglenormElement{\vhat_h - v_h}^2  
    \leq  C h_K \lranglenormElement{\vhat_h - v_h}^2 .
    $$
    Since $u_h \mapsto u_h - u_\partial$ vanishes when
    $u_h \in \calP_0(\calT_h)$, a parent element calculation shows
    $ \norm{u_h - u_\partial}^2_{L^2(K)} \leq C \norm{\nabla
      u_h}^2_{L^2(K)} h^2_K$. Then
    \begin{eqnarray*}
    \norm{\calL_h(\vhat_h) - v_h}_{L^2(K)}^2
    &\leq& C \left(\norm{(\calL_h(\vhat_h) - v_h)
           - (\calL_h(\vhat_h) - v_h)_\partial}^2_{L^2(K)}
           + \norm{\left(\calL_h(\vhat_h) - v_h \right)_\partial}_{L^2(K)}^2
           \vph \right) \\
    &\leq & C\left(
            h_K^2 \norm{\nabla \big(\calL_h(\vhat_h) - v_h)}^2_{L^2(K)}
            + h_K \lranglenormElement{\vhat_h - v_h}^2 
            \right)  \\
    &\leq & C\left(
            h_K^2 \norm{\nabla \calL_h(\vhat_h)}^2_{L^2(K)}
            + h_K^2 \norm{\nabla v_h}^2_{L^2(K)}
            + h_K \lranglenormElement{\vhat_h - v_h}^2 
            \right)  \\
    &\leq& C 
           \left(h_K ^2 \norm{\mbfG_h (v,\vhat)}^2_{L^2(K)}
           + h_K \lranglenormElement{\vhat_h - v_h}^2 \vph\right),
  \end{eqnarray*}
  where the last line follows from the estimates in steps (1) and (2).

\item The Crouzeix Raviart lifting can be estimated using the
  bounds from steps (1) and (3),
  \begin{eqnarray*}
    \norm{\calL_h(\vhat_h)}_{H^1(\calT_h)}
    &\leq& \ltwoh{\calL_h(\vhat_h)} + \ltwoh{\nabla \calL_h(\vhat_h)} \\
    &\leq& \ltwo{v_h - \calL_h(\vhat_h)} + \ltwo{v_h}
           + \ltwo{\mbfG_h(v_h,\vhat_h)} \\
    &\leq& C\left( \ltwo{v_h}
           + \ltwo{\mbfG_h(v_h,\vhat_h)}
           + \sqrt{h} \lranglenorm{\vhat_h - v_h}^{1/2}  \vph\right).
  \end{eqnarray*}
\end{enumerate}
\end{proof}

\begin{proof}[\textbf{Proof of Lemma \ref{lem:HDGproperties}}]
  \begin{enumerate}
  \item (Embedding) Fix $p = 2d/(d-1)$ and let
    $(u_h, \uhat_h) \in \calPell(\calT_h) \times \calPell(\delTh)$.
    Using the inverse estimate for $L^p$ norms on $\calPell(K)$
    \cite[Lemma 4.5.3]{BrSc08}, we compute
    \begin{eqnarray*}
      \norm{u_h}_{L^p(K)}
      &\leq& \norm{u_h - \calL_h(\uhat_h)}_{L^p(K)}
             + \norm{\calL_h(\uhat_h)}_{L^p(K)} \\
      &\leq& C \left(
             h_K^{\frac{d}{p} - \frac{d}{2}}
             \norm{u_h - \calL_h(\uhat_h)}_{L^2(K)}
             + \norm{\calL_h(\uhat_h)}_{L^p(K)} \vph\right) \\
      &\leq& C \Big(
             h_K^{\frac{1}{2} + \frac{d}{p} - \frac{d}{2}}
             \left(\sqrt{h_K} \norm{\bfG_h(u_h,\uhat_h)}_{L^2(K)}
             + \lranglenorm{ \uhat_h - u_h}^{1/2}_{\partial K} \vph\right)
             + \norm{\calL_h(\uhat_h)}_{L^p(K)} \Big) \\
      &\leq& C \Big(
             \norm{\bfG_h(u_h,\uhat_h)}_{L^2(K)}
             + \lranglenorm{ \uhat_h - u_h}^{1/2}_{\partial K} \vph
             + \norm{\calL_h(\uhat_h)}_{L^p(K)} \Big).
    \end{eqnarray*}
    where the second to last line follows from the bounds in Lemma
    \ref{lem:Ghproperties} and the last line following since
    $p = 2d/(d-1)$. If $p \geq 2$ then $(\sum_K a_K^p)^{1/p}
    \leq (\sum_K a_K^2)^{1/2}$ so
    \begin{eqnarray*}
      \lp{u_h}
      &\leq& C \left(
             \ltwo{\bfG_h(u_h,\uhat_h)}
             + \lranglenorm{\uhat_h - u_h}^{1/2}
             + \lp{\calL_h(\uhat_h)} \vph\right) \\
      &\leq& C \left(
             \ltwo{\bfG_h(u_h,\uhat_h)}
             + \lranglenorm{\uhat_h - u_h}^{1/2}
             + \norm{\calL_h(\uhat_h)}_{H^1(\calT_h)} \vph\right) \\
      &\leq& C \left(\ltwo{\bfG_h(u_h,\uhat_h)}
             + \lranglenorm{\uhat_h - u_h}^{1/2}
             + \ltwo{u_h} \vph\right),
    \end{eqnarray*}
    where the last two lines follow since the Crouzeix Raviart
    space embeds into $\Lp$, and the fourth bound in Lemma 
    \ref{lem:Ghproperties}.
    
    To bound the trace of $\vhat_h \in \calPell(\delTh)$, let
    $k \in K \cap \partial \Omega$ be a boundary face, and
    bound the trace on $k$ by
    \begin{eqnarray*}
      \norm{\uhat_h}_{L^2(k)}
      &\leq& \norm{\uhat_h- u_h}_{L^2(k)}
             + \norm{u_h - \calL_h(\uhat_h)}_{L^2(k)} 
             + \norm{\calL_h(\uhat_h)}_{L^2(k)} \\
      &\leq& \norm{\uhat_h- u_h}_{L^2(k)}
             + (C / \sqrt{h_K}) \norm{u_h - \calL_h(\uhat_h)}_{L^2(K)} 
             + \norm{\calL_h(\uhat_h)}_{L^2(k)},
    \end{eqnarray*}
    which follows from the inverse estimate for the trace on $K$.
    Using Lemma \ref{lem:Ghproperties} to bound the middle term
    and summing over $k \subset \partial \Omega$ shows
    \begin{eqnarray*}
      \norm{\uhat_h}_{L^2(\partial\Omega)}
      &\leq& C \left(\lranglenorm{ \uhat_h - u_h}^{1/2}
             + \sqrt{h} \ltwo{\bfG_h(u_h,\uhat_h)} \vph \right)
             + \norm{\calL_h(\uhat_h)}_{L^2(\partial\Omega)} \\
      &\leq& C \left(\lranglenorm{ \uhat_h - u_h}^{1/2}
             + \sqrt{h} \ltwo{\bfG_h(u_h,\uhat_h)} \vph 
             + \norm{\calL_h(\uhat_h)}_{H^1(\calT_h)} \right)\\
      &\leq& C \left(\lranglenorm{ \uhat_h - u_h}^{1/2}
             + \ltwo{\bfG_h(u_h,\uhat_h)}
             + \ltwo{u_h}  \vph \right),
    \end{eqnarray*}
    where the last two lines follow from the trace inequality
    for the Crouzeix Raviart space and the fourth bound
    in lemma \ref{lem:Ghproperties}.
   
  \item (Poincar\'e) When $\uhat_h$ vanishes on $\Gamma_0$, a Poincar\'e
    inequality for the HDG space can be inherited from the Crouzeix
    Raviart space. 
    \begin{eqnarray*}              
      \ltwo{u_h}
      &\leq& \ltwo{u_h - \calL_h(\uhat_h)} + \ltwo{\calL_h(\uhat_h)} \\
      &\leq& C\left(
             \sqrt{h} \ltwo{\bfG_h(u_h,\uhat_h)}
             + \lranglenorm{ \uhat_h - u_h}^{1/2}
             + \ltwoh{\nabla \calL_h(\uhat_h)} \vph\right) \\
      &\leq& C\left(
             \ltwo{\bfG_h(u_h,\uhat_h)}
             + \lranglenorm{ \uhat_h - u_h}^{1/2} \vph\right),
    \end{eqnarray*}
    where the last two lines follow from the third and first bound in
    Lemma \ref{lem:Ghproperties}.

  \item (Compactness) If
    $\{\ltwo{u_h} + \norm{(u_h,\uhat_h)}_{U_h}\}_{h > 0}$ is bounded, it
    follows from Lemma \ref{lem:Ghproperties} that
    $\{\norm{\calL_h(\uhat_h)}_{H^1(\calT_h)} \}_{h > 0}$ is also
    bounded. The compactness properties of the Crouzeix Raviart
    space given in Lemma \ref{lem:CRproperties} allow passage to
    a subsequence and $u \in \Hone$ for which
    $$
    \calL_h(\uhat_h) \rightarrow u \qquad \text{ in } \Ltwo.
    $$
    The triangle inequality and third bound in Lemma
    \ref{lem:Ghproperties} show $u_h \rightarrow u$ in $\Ltwo$

    Since $\{\ltwo{\mbfG_h(u_h,\uhat_h)}\}_{h > 0}$
    and$\{\norm{\uhat_h}_{L^2(\partial\Omega)}\}_{h > 0}$
    are bounded, upon passage to a subsequence they have weak
    limits,
    $$
    \mbfG_h(u_h,\uhat_h) \weak \bfg \quad \text{ in } \Ltwo^d,
    \quad \text{ and } \quad
    \uhat_h \weak \ut \quad \text{ in } L^2(\partial\Omega).
    $$
    To identify the limits, fix
    $\bfq \in C^\infty(\Omegabar)^d$ and select
    $\bfq_h \in C(\bar \Omega)^d \cap \pmb{P}^k(\calT_h)^d$ 
    such that $\norm{\bfq_h - \bfq}_{H^1(\calT_h)} \rightarrow 0$. Then
    \begin{eqnarray*}
      (\bfg,\bfq)
      &=& \lim_{h \rightarrow 0} (\mbfG_h(u_h,\uhat_h), \bfq_h) \\
      &\equiv& \lim_{h \rightarrow 0} \left( (\nabla u_{h}, \bfq_h)
        + \langle \uhat_{h} - u_{h}, \bfq_h \cdot \bfn \rangle \vph\right) \\
      &=& \lim_{h \rightarrow 0} \left( - (u_h, \rmdiv(\bfq_h))
        + \langle \uhat_h, \bfq_h \cdot \bfn \rangle \vph\right) \\
      &=& \lim_{h \rightarrow 0} \left( - (u_h, \rmdiv(\bfq_h))
        + \langle \uhat_h, \bfq_h \cdot \bfn \rangle_{\partial \Omega} \right)
      \\
      &=& - (u, \rmdiv(\bfq))
        + \langle \ut, \bfq \cdot \bfn \rangle_{\partial \Omega}.
    \end{eqnarray*}
    Note that all of the internal jumps vanished since $\bfq_h$ is
    continuous.
    It follows that $(\bfg, \bfq) = -(u,\rmdiv(\bfq))$ for
    all  $\bfq \in C^\infty_c(\Omega)$ whence $\bfg = \nabla u$.
    The above then shows
    $$
    (\nabla u, \bfq) + (u, \rmdiv(\bfq)) = (\ut, \bfq.\bfn)_{\partial\Omega},
    \qquad \bfq \in C^\infty(\Omegabar),
    $$
    so $\ut = u|_{\partial\Omega}$.
  \end{enumerate}
\end{proof}

\bibliography{reference}

@article {CoBeKa05, mr = MR2136994,
    AUTHOR = {Cockburn, Bernardo and Kanschat, Guido and Schotzau, Dominik},
     TITLE = {A locally conservative {LDG} method for the incompressible
              {N}avier-{S}tokes equations},
   JOURNAL = {Math. Comp.},
  FJOURNAL = {Mathematics of Computation},
    VOLUME = {74},
      YEAR = {2005},
    NUMBER = {251},
     PAGES = {1067--1095},
      ISSN = {0025-5718,1088-6842},
   MRCLASS = {65N30 (76D05 76M10)},
  MRNUMBER = {2136994},
MRREVIEWER = {Erik\ Burman},
       DOI = {10.1090/S0025-5718-04-01718-1},
       URL = {https://doi.org/10.1090/S0025-5718-04-01718-1},
}

@incollection {Co00, mr = MR1842161,
    AUTHOR = {Cockburn, Bernardo and Karniadakis, George E. and Shu,
              Chi-Wang},
     TITLE = {The development of discontinuous {G}alerkin methods},
 BOOKTITLE = {Discontinuous {G}alerkin methods ({N}ewport, {RI}, 1999)},
    SERIES = {Lect. Notes Comput. Sci. Eng.},
    VOLUME = {11},
     PAGES = {3--50},
 PUBLISHER = {Springer, Berlin},
      YEAR = {2000},
      ISBN = {3-540-66787-3},
   MRCLASS = {65-02 (65M60 65N30 76M25)},
  MRNUMBER = {1842161},
       DOI = {10.1007/978-3-642-59721-3\_1},
       URL = {https://doi.org/10.1007/978-3-642-59721-3_1},
}

@article {BeMaHeRo12, mr = MR2948694,
    AUTHOR = {Bernauer, Martin K. and Herzog, Roland},
     TITLE = {Implementation of an {X}-{FEM} solver for the classical
              two-phase {S}tefan problem},
   JOURNAL = {J. Sci. Comput.},
  FJOURNAL = {Journal of Scientific Computing},
    VOLUME = {52},
      YEAR = {2012},
    NUMBER = {2},
     PAGES = {271--293},
      ISSN = {0885-7474,1573-7691},
   MRCLASS = {65M60 (35K10 35R35 80A22)},
  MRNUMBER = {2948694},
MRREVIEWER = {Subbiah\ Sundar},
       DOI = {10.1007/s10915-011-9543-x},
       URL = {https://doi.org/10.1007/s10915-011-9543-x},
}

@article {Hu20, mr = MR4057428,
    AUTHOR = {Hu, Weiwei and Shen, Jiguang and Singler, John R. and Zhang,
              Yangwen and Zheng, Xiaobo},
     TITLE = {A superconvergent hybridizable discontinuous {G}alerkin method
              for {D}irichlet boundary control of elliptic {PDE}s},
   JOURNAL = {Numer. Math.},
  FJOURNAL = {Numerische Mathematik},
    VOLUME = {144},
      YEAR = {2020},
    NUMBER = {2},
     PAGES = {375--411},
      ISSN = {0029-599X,0945-3245},
   MRCLASS = {65N30 (49M25 49M40)},
  MRNUMBER = {4057428},
MRREVIEWER = {Ruchi\ Sandilya},
       DOI = {10.1007/s00211-019-01090-2},
       URL = {https://doi.org/10.1007/s00211-019-01090-2},
}

@article {FuQi15, mr = MR3342199,
    AUTHOR = {Fu, Guosheng and Qiu, Weifeng and Zhang, Wujun},
     TITLE = {An analysis of {HDG} methods for convection-dominated
              diffusion problems},
   JOURNAL = {ESAIM Math. Model. Numer. Anal.},
  FJOURNAL = {ESAIM. Mathematical Modelling and Numerical Analysis},
    VOLUME = {49},
      YEAR = {2015},
    NUMBER = {1},
     PAGES = {225--256},
      ISSN = {2822-7840,2804-7214},
   MRCLASS = {65N30 (65N12 65N22)},
  MRNUMBER = {3342199},
MRREVIEWER = {Xavier\ Claeys},
       DOI = {10.1051/m2an/2014032},
       URL = {https://doi.org/10.1051/m2an/2014032},
}

@article {CoSh98, mr = MR1655854,
    AUTHOR = {Cockburn, Bernardo and Shu, Chi-Wang},
     TITLE = {The local discontinuous {G}alerkin method for time-dependent
              convection-diffusion systems},
   JOURNAL = {SIAM J. Numer. Anal.},
  FJOURNAL = {SIAM Journal on Numerical Analysis},
    VOLUME = {35},
      YEAR = {1998},
    NUMBER = {6},
     PAGES = {2440--2463},
      ISSN = {0036-1429,1095-7170},
   MRCLASS = {65M60 (76M25)},
  MRNUMBER = {1655854},
MRREVIEWER = {Leonid\ K.\ Antanovski\u{\i}},
       DOI = {10.1137/S0036142997316712},
       URL = {https://doi.org/10.1137/S0036142997316712},
}

@book {HeWa08, mr = MR2372235,
    AUTHOR = {Hesthaven, Jan S. and Warburton, Tim},
     TITLE = {Nodal discontinuous {G}alerkin methods},
    SERIES = {Texts in Applied Mathematics},
    VOLUME = {54},
      NOTE = {Algorithms, analysis, and applications},
 PUBLISHER = {Springer, New York},
      YEAR = {2008},
     PAGES = {xiv+500},
      ISBN = {978-0-387-72065-4},
   MRCLASS = {65-02 (65M60 65N30)},
  MRNUMBER = {2372235},
MRREVIEWER = {Weimin\ Han},
       DOI = {10.1007/978-0-387-72067-8},
       URL = {https://doi.org/10.1007/978-0-387-72067-8},
}

@article {Br03, mr = MR1974504,
    AUTHOR = {Brenner, Susanne C.},
     TITLE = {Poincar\'{e}-{F}riedrichs inequalities for piecewise {$H^1$}
              functions},
   JOURNAL = {SIAM J. Numer. Anal.},
  FJOURNAL = {SIAM Journal on Numerical Analysis},
    VOLUME = {41},
      YEAR = {2003},
    NUMBER = {1},
     PAGES = {306--324},
      ISSN = {0036-1429,1095-7170},
   MRCLASS = {65N30 (46E35)},
  MRNUMBER = {1974504},
       DOI = {10.1137/S0036142902401311},
       URL = {https://doi.org/10.1137/S0036142902401311},
}

@article {CoGoSa10, mr = MR2629996,
    AUTHOR = {Cockburn, Bernardo and Gopalakrishnan, Jayadeep and Sayas,
              Francisco-Javier},
     TITLE = {A projection-based error analysis of {HDG} methods},
   JOURNAL = {Math. Comp.},
  FJOURNAL = {Mathematics of Computation},
    VOLUME = {79},
      YEAR = {2010},
    NUMBER = {271},
     PAGES = {1351--1367},
      ISSN = {0025-5718,1088-6842},
   MRCLASS = {65N30 (65N15)},
  MRNUMBER = {2629996},
MRREVIEWER = {Mohammad\ Asadzadeh},
       DOI = {10.1090/S0025-5718-10-02334-3},
       URL = {https://doi.org/10.1090/S0025-5718-10-02334-3},
}

@article {nguyen2010hybridizable, mr = MR2796169,
    AUTHOR = {Nguyen, N. C. and Peraire, J. and Cockburn, B.},
     TITLE = {A hybridizable discontinuous {G}alerkin method for {S}tokes
              flow},
   JOURNAL = {Comput. Methods Appl. Mech. Engrg.},
  FJOURNAL = {Computer Methods in Applied Mechanics and Engineering},
    VOLUME = {199},
      YEAR = {2010},
    NUMBER = {9-12},
     PAGES = {582--597},
      ISSN = {0045-7825,1879-2138},
   MRCLASS = {65N30 (76D07 76M10)},
  MRNUMBER = {2796169},
       DOI = {10.1016/j.cma.2009.10.007},
       URL = {https://doi.org/10.1016/j.cma.2009.10.007},
}

@incollection {cockburn2016static, mr = MR3585789,
    AUTHOR = {Cockburn, Bernardo},
     TITLE = {Static condensation, hybridization, and the devising of the
              {HDG} methods},
 BOOKTITLE = {Building bridges: connections and challenges in modern
              approaches to numerical partial differential equations},
    SERIES = {Lect. Notes Comput. Sci. Eng.},
    VOLUME = {114},
     PAGES = {129--177},
 PUBLISHER = {Springer, [Cham]},
      YEAR = {2016},
      ISBN = {978-3-319-41638-0; 978-3-319-41640-3},
   MRCLASS = {65N30 (35L50)},
  MRNUMBER = {3585789},
}

@article {castillo2000priori, mr = MR1813251,
    AUTHOR = {Castillo, Paul and Cockburn, Bernardo and Perugia, Ilaria and
              Sch\"{o}tzau, Dominik},
     TITLE = {An a priori error analysis of the local discontinuous
              {G}alerkin method for elliptic problems},
   JOURNAL = {SIAM J. Numer. Anal.},
  FJOURNAL = {SIAM Journal on Numerical Analysis},
    VOLUME = {38},
      YEAR = {2000},
    NUMBER = {5},
     PAGES = {1676--1706},
      ISSN = {0036-1429,1095-7170},
   MRCLASS = {65N15 (65N30)},
  MRNUMBER = {1813251},
       DOI = {10.1137/S0036142900371003},
       URL = {https://doi.org/10.1137/S0036142900371003},
}

@article {kirk2019analysis, mr = MR4019980,
    AUTHOR = {Kirk, Keegan L. A. and Rhebergen, Sander},
     TITLE = {Analysis of a pressure-robust hybridized discontinuous
              {G}alerkin method for the stationary {N}avier-{S}tokes
              equations},
   JOURNAL = {J. Sci. Comput.},
  FJOURNAL = {Journal of Scientific Computing},
    VOLUME = {81},
      YEAR = {2019},
    NUMBER = {2},
     PAGES = {881--897},
      ISSN = {0885-7474,1573-7691},
   MRCLASS = {65N30 (35Q30)},
  MRNUMBER = {4019980},
       DOI = {10.1007/s10915-019-01040-y},
       URL = {https://doi.org/10.1007/s10915-019-01040-y},
}

@article {chen2023superconvergence, mr = MR4543432,
    AUTHOR = {Chen, Gang and Han, Daozhi and Singler, John R. and Zhang,
              Yangwen},
     TITLE = {On the superconvergence of a hybridizable discontinuous
              {G}alerkin method for the {C}ahn-{H}illiard equation},
   JOURNAL = {SIAM J. Numer. Anal.},
  FJOURNAL = {SIAM Journal on Numerical Analysis},
    VOLUME = {61},
      YEAR = {2023},
    NUMBER = {1},
     PAGES = {83--109},
      ISSN = {0036-1429,1095-7170},
   MRCLASS = {65M12 (65M15 65M22 65M60)},
  MRNUMBER = {4543432},
       DOI = {10.1137/21M1437780},
       URL = {https://doi.org/10.1137/21M1437780},
}

@article{ArBrCoMa02,
  title={Unified analysis of discontinuous Galerkin methods for elliptic problems},
  author={Arnold, Douglas N and Brezzi, Franco and Cockburn, Bernardo and Marini, Donatella},
  journal={SIAM journal on numerical analysis},
  volume={39},
  number={5},
  pages={1749--1779},
  year={2002},
  publisher={SIAM}
}

@article {cockburn2009unified, mr = MR2485455,
    AUTHOR = {Cockburn, Bernardo and Gopalakrishnan, Jayadeep and Lazarov,
              Raytcho},
     TITLE = {Unified hybridization of discontinuous {G}alerkin, mixed, and
              continuous {G}alerkin methods for second order elliptic
              problems},
   JOURNAL = {SIAM J. Numer. Anal.},
  FJOURNAL = {SIAM Journal on Numerical Analysis},
    VOLUME = {47},
      YEAR = {2009},
    NUMBER = {2},
     PAGES = {1319--1365},
      ISSN = {0036-1429,1095-7170},
   MRCLASS = {65N30},
  MRNUMBER = {2485455},
MRREVIEWER = {Jose\ Luis\ Gracia},
       DOI = {10.1137/070706616},
       URL = {https://doi.org/10.1137/070706616},
}

@incollection {sevilla2016tutorial, mr = MR3586449,
    AUTHOR = {Sevilla, Ruben and Huerta, Antonio},
     TITLE = {Tutorial on hybridizable discontinuous {G}alerkin ({HDG}) for
              second-order elliptic problems},
 BOOKTITLE = {Advanced finite element technologies},
    SERIES = {CISM Courses and Lect.},
    VOLUME = {566},
     PAGES = {105--129},
 PUBLISHER = {Springer, [Cham]},
      YEAR = {2016},
      ISBN = {978-3-319-31923-0; 978-3-319-31925-4},
   MRCLASS = {65N30 (65N12)},
  MRNUMBER = {3586449},
MRREVIEWER = {Ignacio\ Romero},
}

@article {cockburn2004characterization, mr = MR2051067,
    AUTHOR = {Cockburn, Bernardo and Gopalakrishnan, Jayadeep},
     TITLE = {A characterization of hybridized mixed methods for second
              order elliptic problems},
   JOURNAL = {SIAM J. Numer. Anal.},
  FJOURNAL = {SIAM Journal on Numerical Analysis},
    VOLUME = {42},
      YEAR = {2004},
    NUMBER = {1},
     PAGES = {283--301},
      ISSN = {0036-1429,1095-7170},
   MRCLASS = {65N30},
  MRNUMBER = {2051067},
MRREVIEWER = {S\"{o}ren\ Bartels},
       DOI = {10.1137/S0036142902417893},
       URL = {https://doi.org/10.1137/S0036142902417893},
}

@article {cockburn2012conditions, mr = MR2904581,
    AUTHOR = {Cockburn, Bernardo and Qiu, Weifeng and Shi, Ke},
     TITLE = {Conditions for superconvergence of {HDG} methods for
              second-order elliptic problems},
   JOURNAL = {Math. Comp.},
  FJOURNAL = {Mathematics of Computation},
    VOLUME = {81},
      YEAR = {2012},
    NUMBER = {279},
     PAGES = {1327--1353},
      ISSN = {0025-5718,1088-6842},
   MRCLASS = {65N30 (35J25 65N12)},
  MRNUMBER = {2904581},
MRREVIEWER = {Erich\ Novak},
       DOI = {10.1090/S0025-5718-2011-02550-0},
       URL = {https://doi.org/10.1090/S0025-5718-2011-02550-0},
}

@article {efendiev2015spectral, mr = MR3347248,
    AUTHOR = {Efendiev, Yalchin and Lazarov, Raytcho and Moon, Minam and
              Shi, Ke},
     TITLE = {A spectral multiscale hybridizable discontinuous {G}alerkin
              method for second order elliptic problems},
   JOURNAL = {Comput. Methods Appl. Mech. Engrg.},
  FJOURNAL = {Computer Methods in Applied Mechanics and Engineering},
    VOLUME = {292},
      YEAR = {2015},
     PAGES = {243--256},
      ISSN = {0045-7825,1879-2138},
   MRCLASS = {65N35 (35J25 65N30)},
  MRNUMBER = {3347248},
MRREVIEWER = {Michail\ Dimov\ Todorov},
       DOI = {10.1016/j.cma.2014.09.036},
       URL = {https://doi.org/10.1016/j.cma.2014.09.036},
}

@article {li2016analysis, mr = MR3508837,
    AUTHOR = {Li, Binjie and Xie, Xiaoping},
     TITLE = {Analysis of a family of {HDG} methods for second order
              elliptic problems},
   JOURNAL = {J. Comput. Appl. Math.},
  FJOURNAL = {Journal of Computational and Applied Mathematics},
    VOLUME = {307},
      YEAR = {2016},
     PAGES = {37--51},
      ISSN = {0377-0427,1879-1778},
   MRCLASS = {65N30 (65N12)},
  MRNUMBER = {3508837},
MRREVIEWER = {Christian\ Wieners},
       DOI = {10.1016/j.cam.2016.04.027},
       URL = {https://doi.org/10.1016/j.cam.2016.04.027},
}

@article {buffa2009compact, mr = MR2557047,
    AUTHOR = {Buffa, Annalisa and Ortner, Christoph},
     TITLE = {Compact embeddings of broken {S}obolev spaces and
              applications},
   JOURNAL = {IMA J. Numer. Anal.},
  FJOURNAL = {IMA Journal of Numerical Analysis},
    VOLUME = {29},
      YEAR = {2009},
    NUMBER = {4},
     PAGES = {827--855},
      ISSN = {0272-4979,1464-3642},
   MRCLASS = {65J05 (65N30)},
  MRNUMBER = {2557047},
       DOI = {10.1093/imanum/drn038},
       URL = {https://doi.org/10.1093/imanum/drn038},
}

@book {DaBe08, mr = MR2361288,
    AUTHOR = {Dacorogna, Bernard},
     TITLE = {Direct methods in the calculus of variations},
    SERIES = {Applied Mathematical Sciences},
    VOLUME = {78},
   EDITION = {Second},
 PUBLISHER = {Springer, New York},
      YEAR = {2008},
     PAGES = {xii+619},
      ISBN = {978-0-387-35779-9},
   MRCLASS = {49-02 (49J10 49J45 74B20)},
  MRNUMBER = {2361288},
MRREVIEWER = {Pietro\ Celada},
}

@book {Sho97, mr = MR1422252,
    AUTHOR = {Showalter, R. E.},
     TITLE = {Monotone operators in {B}anach space and nonlinear partial
              differential equations},
    SERIES = {Mathematical Surveys and Monographs},
    VOLUME = {49},
 PUBLISHER = {American Mathematical Society, Providence, RI},
      YEAR = {1997},
     PAGES = {xiv+278},
      ISBN = {0-8218-0500-2},
   MRCLASS = {47H15 (34G20 35J60 35K55 47H05 47N20)},
  MRNUMBER = {1422252},
MRREVIEWER = {Ioan\ I.\ Vrabie},
       DOI = {10.1090/surv/049},
       URL = {https://doi.org/10.1090/surv/049},
}

@book {BrFo91, mr = MR1115205,
    AUTHOR = {Brezzi, Franco and Fortin, Michel},
     TITLE = {Mixed and hybrid finite element methods},
    SERIES = {Springer Series in Computational Mathematics},
    VOLUME = {15},
 PUBLISHER = {Springer-Verlag, New York},
      YEAR = {1991},
     PAGES = {x+350},
      ISBN = {0-387-97582-9},
   MRCLASS = {65N30 (65-02 73V05 76M10)},
  MRNUMBER = {1115205},
MRREVIEWER = {Lubor\ Malina},
       DOI = {10.1007/978-1-4612-3172-1},
       URL = {https://doi.org/10.1007/978-1-4612-3172-1},
}

@book {Ri08, mr = MR2431403,
    AUTHOR = {Rivi\`ere, B\'{e}atrice},
     TITLE = {Discontinuous {G}alerkin methods for solving elliptic and
              parabolic equations},
    SERIES = {Frontiers in Applied Mathematics},
    VOLUME = {35},
      NOTE = {Theory and implementation},
 PUBLISHER = {Society for Industrial and Applied Mathematics (SIAM),
              Philadelphia, PA},
      YEAR = {2008},
     PAGES = {xxii+190},
      ISBN = {978-0-898716-56-6},
   MRCLASS = {65N30 (65M60 76M10 76S05)},
  MRNUMBER = {2431403},
MRREVIEWER = {Colin\ J.\ Cotter},
       DOI = {10.1137/1.9780898717440},
       URL = {https://doi.org/10.1137/1.9780898717440},
}

@book{TemamNavierStokes,
  title={Navier-Stokes equations: theory and numerical analysis},
  author={Temam, Roger},
  volume={343},
  year={2001},
  publisher={American Mathematical Soc.}
}

@book {Gu21, mr=MR4269305,
    AUTHOR = {Ern, Alexandre and Guermond, Jean-Luc},
     TITLE = {Finite elements {II}---{G}alerkin approximation, elliptic and
              mixed {PDE}s},
    SERIES = {Texts in Applied Mathematics},
    VOLUME = {73},
 PUBLISHER = {Springer, Cham},
      YEAR = {[2021] \copyright 2021},
     PAGES = {ix+492},
      ISBN = {978-3-030-56922-8; 978-3-030-56923-5},
   MRCLASS = {65-01},
  MRNUMBER = {4269305},
       DOI = {10.1007/978-3-030-56923-5},
       URL = {https://doi.org/10.1007/978-3-030-56923-5},
}

@article {Ar82, mr = MR0664882,
    AUTHOR = {Arnold, Douglas N.},
     TITLE = {An interior penalty finite element method with discontinuous
              elements},
   JOURNAL = {SIAM J. Numer. Anal.},
  FJOURNAL = {SIAM Journal on Numerical Analysis},
    VOLUME = {19},
      YEAR = {1982},
    NUMBER = {4},
     PAGES = {742--760},
      ISSN = {0036-1429},
   MRCLASS = {65N30},
  MRNUMBER = {664882},
       DOI = {10.1137/0719052},
       URL = {https://doi.org/10.1137/0719052},
}

@article {OdBuBa98, mr = MR1654911,
    AUTHOR = {Oden, J. Tinsley and Babu\v{s}ka, Ivo and Baumann, Carlos
              Erik},
     TITLE = {A discontinuous {$hp$} finite element method for diffusion
              problems},
   JOURNAL = {J. Comput. Phys.},
  FJOURNAL = {Journal of Computational Physics},
    VOLUME = {146},
      YEAR = {1998},
    NUMBER = {2},
     PAGES = {491--519},
      ISSN = {0021-9991,1090-2716},
   MRCLASS = {65M12 (65M60 76M10)},
  MRNUMBER = {1654911},
MRREVIEWER = {Gert\ Lube},
       DOI = {10.1006/jcph.1998.6032},
       URL = {https://doi.org/10.1006/jcph.1998.6032},
}

@article {RiWh99, mr = MR1750076,
    AUTHOR = {Rivi\`ere, B\'{e}atrice and Wheeler, Mary F. and Girault,
              Vivette},
     TITLE = {Improved energy estimates for interior penalty, constrained
              and discontinuous {G}alerkin methods for elliptic problems.
              {I}},
   JOURNAL = {Comput. Geosci.},
  FJOURNAL = {Computational Geosciences},
    VOLUME = {3},
      YEAR = {1999},
    NUMBER = {3-4},
     PAGES = {337--360},
      ISSN = {1420-0597,1573-1499},
   MRCLASS = {65N15 (65N30)},
  MRNUMBER = {1750076},
MRREVIEWER = {G.\ Albinus},
       DOI = {10.1023/A:1011591328604},
       URL = {https://doi.org/10.1023/A:1011591328604},
}

@article {Br04, mr = MR2047078,
    AUTHOR = {Brenner, Susanne C.},
     TITLE = {Korn's inequalities for piecewise {$H^1$} vector fields},
   JOURNAL = {Math. Comp.},
  FJOURNAL = {Mathematics of Computation},
    VOLUME = {73},
      YEAR = {2004},
    NUMBER = {247},
     PAGES = {1067--1087},
      ISSN = {0025-5718,1088-6842},
   MRCLASS = {65N30 (26D15 35J55 74S05)},
  MRNUMBER = {2047078},
MRREVIEWER = {Thomas\ Apel},
       DOI = {10.1090/S0025-5718-03-01579-5},
       URL = {https://doi.org/10.1090/S0025-5718-03-01579-5},
}

@article {BaNa97, mr = R1423460,
    AUTHOR = {Babu\v{s}ka, I. and Narasimhan, R.},
     TITLE = {The {B}abu\v{s}ka-{B}rezzi condition and the patch test: an
              example},
   JOURNAL = {Comput. Methods Appl. Mech. Engrg.},
  FJOURNAL = {Computer Methods in Applied Mechanics and Engineering},
    VOLUME = {140},
      YEAR = {1997},
    NUMBER = {1-2},
     PAGES = {183--199},
      ISSN = {0045-7825,1879-2138},
   MRCLASS = {73V05 (65N30)},
  MRNUMBER = {1423460},
MRREVIEWER = {Zhimin\ Zhang},
       DOI = {10.1016/S0045-7825(96)01058-4},
       URL = {https://doi.org/10.1016/S0045-7825(96)01058-4},
}

@article {FeLeNe16, mr = MR3452848,
    AUTHOR = {Feng, Xiaobing and Lewis, Thomas and Neilan, Michael},
     TITLE = {Discontinuous {G}alerkin finite element differential calculus
              and applications to numerical solutions of linear and
              nonlinear partial differential equations},
   JOURNAL = {J. Comput. Appl. Math.},
  FJOURNAL = {Journal of Computational and Applied Mathematics},
    VOLUME = {299},
      YEAR = {2016},
     PAGES = {68--91},
      ISSN = {0377-0427,1879-1778},
   MRCLASS = {65D25 (65N30)},
  MRNUMBER = {3452848},
       DOI = {10.1016/j.cam.2015.10.024},
       URL = {https://doi.org/10.1016/j.cam.2015.10.024},
}

@article {WaYe13, mr = MR2994424,
    AUTHOR = {Wang, Junping and Ye, Xiu},
     TITLE = {A weak {G}alerkin finite element method for second-order
              elliptic problems},
   JOURNAL = {J. Comput. Appl. Math.},
  FJOURNAL = {Journal of Computational and Applied Mathematics},
    VOLUME = {241},
      YEAR = {2013},
     PAGES = {103--115},
      ISSN = {0377-0427,1879-1778},
   MRCLASS = {65N30 (65N15)},
  MRNUMBER = {2994424},
       DOI = {10.1016/j.cam.2012.10.003},
       URL = {https://doi.org/10.1016/j.cam.2012.10.003},
}

@article {WaYe14, mr = MR3223326,
    AUTHOR = {Wang, Junping and Ye, Xiu},
     TITLE = {A weak {G}alerkin mixed finite element method for second order
              elliptic problems},
   JOURNAL = {Math. Comp.},
  FJOURNAL = {Mathematics of Computation},
    VOLUME = {83},
      YEAR = {2014},
    NUMBER = {289},
     PAGES = {2101--2126},
      ISSN = {0025-5718,1088-6842},
   MRCLASS = {65N30 (35B45 35J57 35Q35 65N15)},
  MRNUMBER = {3223326},
MRREVIEWER = {Snorre\ H.\ Christiansen},
       DOI = {10.1090/S0025-5718-2014-02852-4},
       URL = {https://doi.org/10.1090/S0025-5718-2014-02852-4},
}

@book {Ci78, mr = MR0520174,
    AUTHOR = {Ciarlet, Philippe G.},
     TITLE = {The finite element method for elliptic problems},
    SERIES = {Studies in Mathematics and its Applications},
    VOLUME = {Vol. 4},
 PUBLISHER = {North-Holland Publishing Co., Amsterdam-New York-Oxford},
      YEAR = {1978},
     PAGES = {xix+530},
      ISBN = {0-444-85028-7},
   MRCLASS = {65N30},
  MRNUMBER = {520174},
MRREVIEWER = {Josef\ Nedoma},
}

@book {BrSc08, mr = MR2373954,
    AUTHOR = {Brenner, Susanne C. and Scott, L. Ridgway},
     TITLE = {The mathematical theory of finite element methods},
    SERIES = {Texts in Applied Mathematics},
    VOLUME = {15},
   EDITION = {Third},
 PUBLISHER = {Springer, New York},
      YEAR = {2008},
     PAGES = {xviii+397},
      ISBN = {978-0-387-75933-3},
   MRCLASS = {65-01 (65-02)},
  MRNUMBER = {2373954},
       DOI = {10.1007/978-0-387-75934-0},
       URL = {https://doi.org/10.1007/978-0-387-75934-0},
}

@book {ErGu04, mr = MR2050138,
    AUTHOR = {Ern, Alexandre and Guermond, Jean-Luc},
     TITLE = {Theory and practice of finite elements},
    SERIES = {Applied Mathematical Sciences},
    VOLUME = {159},
 PUBLISHER = {Springer-Verlag, New York},
      YEAR = {2004},
     PAGES = {xiv+524},
      ISBN = {0-387-20574-8},
   MRCLASS = {65-02 (65M60 65N30 74S05 76M10 78M10)},
  MRNUMBER = {2050138},
MRREVIEWER = {R.\ S.\ Anderssen},
       DOI = {10.1007/978-1-4757-4355-5},
       URL = {https://doi.org/10.1007/978-1-4757-4355-5},
}

@article {RuWa96, mr=MR1377243,
    AUTHOR = {Rulla, Jim and Walkington, Noel J.},
     TITLE = {Optimal rates of convergence for degenerate parabolic problems
              in two dimensions},
   JOURNAL = {SIAM J. Numer. Anal.},
  FJOURNAL = {SIAM Journal on Numerical Analysis},
    VOLUME = {33},
      YEAR = {1996},
    NUMBER = {1},
     PAGES = {56--67},
      ISSN = {0036-1429},
   MRCLASS = {65M12},
  MRNUMBER = {1377243},
MRREVIEWER = {Rudolf\ Gorenflo},
       DOI = {10.1137/0733004},
       URL = {https://doi.org/10.1137/0733004},
}

@article {chung2020generalized, mr = MR4079470,
    AUTHOR = {Chung, Eric and Efendiev, Yalchin and Li, Yanbo and Li, Qin},
     TITLE = {Generalized multiscale finite element method for the steady
              state linear {B}oltzmann equation},
   JOURNAL = {Multiscale Model. Simul.},
  FJOURNAL = {Multiscale Modeling \& Simulation. A SIAM Interdisciplinary
              Journal},
    VOLUME = {18},
      YEAR = {2020},
    NUMBER = {1},
     PAGES = {475--501},
      ISSN = {1540-3459,1540-3467},
   MRCLASS = {65N30 (35Q20 45K05 65N12)},
  MRNUMBER = {4079470},
MRREVIEWER = {Weifeng\ Qiu},
       DOI = {10.1137/19M1256282},
       URL = {https://doi.org/10.1137/19M1256282},
}

@article {du2020unified, mr = MR4092329,
    AUTHOR = {Du, Shukai and Sayas, Francisco-Javier},
     TITLE = {A unified error analysis of hybridizable discontinuous
              {G}alerkin methods for the static {M}axwell equations},
   JOURNAL = {SIAM J. Numer. Anal.},
  FJOURNAL = {SIAM Journal on Numerical Analysis},
    VOLUME = {58},
      YEAR = {2020},
    NUMBER = {2},
     PAGES = {1367--1391},
      ISSN = {0036-1429,1095-7170},
   MRCLASS = {65N30 (35Q61 65N15 78A25)},
  MRNUMBER = {4092329},
MRREVIEWER = {J\"{u}rgen\ D\"{o}lz},
       DOI = {10.1137/19M1290966},
       URL = {https://doi.org/10.1137/19M1290966},
}

@article {cockburn2004locally, mr = MR2034859,
    AUTHOR = {Cockburn, Bernardo and Li, Fengyan and Shu, Chi-Wang},
     TITLE = {Locally divergence-free discontinuous {G}alerkin methods for
              the {M}axwell equations},
   JOURNAL = {J. Comput. Phys.},
  FJOURNAL = {Journal of Computational Physics},
    VOLUME = {194},
      YEAR = {2004},
    NUMBER = {2},
     PAGES = {588--610},
      ISSN = {0021-9991,1090-2716},
   MRCLASS = {78M10 (65M60)},
  MRNUMBER = {2034859},
       DOI = {10.1016/j.jcp.2003.09.007},
       URL = {https://doi.org/10.1016/j.jcp.2003.09.007},
}

@article {jaiswal2019discontinuous, mr = MR3881584,
    AUTHOR = {Jaiswal, Shashank and Alexeenko, Alina A. and Hu, Jingwei},
     TITLE = {A discontinuous {G}alerkin fast spectral method for the full
              {B}oltzmann equation with general collision kernels},
   JOURNAL = {J. Comput. Phys.},
  FJOURNAL = {Journal of Computational Physics},
    VOLUME = {378},
      YEAR = {2019},
     PAGES = {178--208},
      ISSN = {0021-9991,1090-2716},
   MRCLASS = {65M70 (45K05 65M60 76P05)},
  MRNUMBER = {3881584},
MRREVIEWER = {Mohsen\ Esmaeilbeigi},
       DOI = {10.1016/j.jcp.2018.11.001},
       URL = {https://doi.org/10.1016/j.jcp.2018.11.001},
}

@article {cheng2011brief, mr = MR2839296,
    AUTHOR = {Cheng, Yingda and Gamba, Irene M. and Majorana, Armando and
              Shu, Chi-Wang},
     TITLE = {A brief survey of the discontinuous {G}alerkin method for the
              {B}oltzmann-{P}oisson equations},
   JOURNAL = {SeMA J.},
  FJOURNAL = {SeMA Journal},
    NUMBER = {54},
      YEAR = {2011},
     PAGES = {47--64},
      ISSN = {1575-9822},
   MRCLASS = {82C70 (35Q20 65M60)},
  MRNUMBER = {2839296},
MRREVIEWER = {S\'{e}bastien\ J.\ Boyaval},
       DOI = {10.1007/bf03322587},
       URL = {https://doi.org/10.1007/bf03322587},
}

@article {kay2009discontinuous, mr = MR2551141,
    AUTHOR = {Kay, David and Styles, Vanessa and S\"{u}li, Endre},
     TITLE = {Discontinuous {G}alerkin finite element approximation of the
              {C}ahn-{H}illiard equation with convection},
   JOURNAL = {SIAM J. Numer. Anal.},
  FJOURNAL = {SIAM Journal on Numerical Analysis},
    VOLUME = {47},
      YEAR = {2009},
    NUMBER = {4},
     PAGES = {2660--2685},
      ISSN = {0036-1429,1095-7170},
   MRCLASS = {65M60 (65M15)},
  MRNUMBER = {2551141},
MRREVIEWER = {Alan\ Demlow},
       DOI = {10.1137/080726768},
       URL = {https://doi.org/10.1137/080726768},
}

@article {xia2007local, mr = MR2361532,
    AUTHOR = {Xia, Yinhua and Xu, Yan and Shu, Chi-Wang},
     TITLE = {Local discontinuous {G}alerkin methods for the
              {C}ahn-{H}illiard type equations},
   JOURNAL = {J. Comput. Phys.},
  FJOURNAL = {Journal of Computational Physics},
    VOLUME = {227},
      YEAR = {2007},
    NUMBER = {1},
     PAGES = {472--491},
      ISSN = {0021-9991,1090-2716},
   MRCLASS = {65M60 (35K55 35Q53)},
  MRNUMBER = {2361532},
       DOI = {10.1016/j.jcp.2007.08.001},
       URL = {https://doi.org/10.1016/j.jcp.2007.08.001},
}

@book {efendiev2009multiscale, mr = MR2477579,
    AUTHOR = {Efendiev, Yalchin and Hou, Thomas Y.},
     TITLE = {Multiscale finite element methods},
    SERIES = {Surveys and Tutorials in the Applied Mathematical Sciences},
    VOLUME = {4},
      NOTE = {Theory and applications},
 PUBLISHER = {Springer, New York},
      YEAR = {2009},
     PAGES = {xii+234},
      ISBN = {978-0-387-09495-3},
   MRCLASS = {65N30 (74Q05 74S05 76M10 76M45)},
  MRNUMBER = {2477579},
MRREVIEWER = {Stefan\ Henn},
}

@article {efendiev2022efficient, mr = MR4460336,
    AUTHOR = {Efendiev, Yalchin and Leung, Wing Tat and Lin, Guang and
              Zhang, Zecheng},
     TITLE = {Efficient hybrid explicit-implicit learning for multiscale
              problems},
   JOURNAL = {J. Comput. Phys.},
  FJOURNAL = {Journal of Computational Physics},
    VOLUME = {467},
      YEAR = {2022},
     PAGES = {Paper No. 111326, 15},
      ISSN = {0021-9991,1090-2716},
   MRCLASS = {65M99 (68T07)},
  MRNUMBER = {4460336},
       DOI = {10.1016/j.jcp.2022.111326},
       URL = {https://doi.org/10.1016/j.jcp.2022.111326},
}

@article {hou2015sparse, mr = MR3320204,
    AUTHOR = {Hou, Thomas Y. and Li, Qin and Schaeffer, Hayden},
     TITLE = {Sparse + low-energy decomposition for viscous conservation
              laws},
   JOURNAL = {J. Comput. Phys.},
  FJOURNAL = {Journal of Computational Physics},
    VOLUME = {288},
      YEAR = {2015},
     PAGES = {150--166},
      ISSN = {0021-9991,1090-2716},
   MRCLASS = {65M99 (76Nxx)},
  MRNUMBER = {3320204},
       DOI = {10.1016/j.jcp.2015.02.019},
       URL = {https://doi.org/10.1016/j.jcp.2015.02.019},
}

\end{document}